\documentclass[twoside,11pt,reqno]{amsart}
\usepackage{amsmath,amssymb,amscd,mathrsfs,amscd}
\usepackage{graphics,verbatim}
\usepackage{hyperref}
\usepackage[all]{xy}
\usepackage{tikz}
\usepackage{tikz-cd}

\let\oldsection=\section

\usepackage[left=1in, top=1in, right=1in, bottom=1in]{geometry}

\usetikzlibrary{decorations.markings}
\usetikzlibrary{decorations.pathreplacing}
\usetikzlibrary{arrows,shapes,positioning}
\tikzstyle directed=[postaction={decorate,decoration={markings, mark=at position #1 with {\arrow[scale=1]{>}}}}]
\tikzstyle rdirected=[postaction={decorate,decoration={markings, mark=at position #1 with {\arrow[scale=1]{<}}}}]
\usepackage[all]{xy}
\newcommand{\clr}{rgb:black,1;blue,4;red,1}

\newcommand{\bdot}{ node[circle, draw, fill=\clr, thick, inner sep=0pt, minimum width=6pt]{}}
\newcommand{\tick}{ node[rectangle, draw, fill=\clr,  inner sep=0pt, minimum width=1pt, minimum height=6pt]{}}

\newcommand{\p}[1]{\ensuremath{\overline{#1}}}
\newcommand{\losemi}{{\otimes \kern -.78em \ltimes}}
\newcommand{\rosemi}{{\otimes \kern -.78em \rtimes}}

\newcommand{\Hom}{\ensuremath{\operatorname{Hom}}}

\newcommand{\Ext}{\operatorname{Ext}}

\newcommand{\0}{\bar 0}
\renewcommand{\1}{\bar 1}
\newcommand{\Z}{\mathbb{Z}}
\newcommand{\C}{\mathbb{C}}

\newcommand{\gl}{\ensuremath{\mathfrak{gl}}}
\newcommand{\g}{\ensuremath{\mathfrak{g}}}

\newcommand{\f}{\ensuremath{\mathfrak{f}}}

\newcommand{\fg}{\ensuremath{\mathfrak{g}}}
\newcommand{\fb}{\ensuremath{\mathfrak{b}}}

\newcommand{\fh}{\ensuremath{\mathfrak{h}}}
\newcommand{\ff}{\ensuremath{\f}}

\newcommand{\fc}{\ensuremath{\mathfrak{c}}}

\newcommand{\fp}{\ensuremath{\mathfrak{p}}}

\renewcommand{\a}{\alpha}

\newcommand{\ve}{\varepsilon}

\newcommand{\atyp}{\ensuremath{\operatorname{atyp}}}
\newcommand{\V}{\mathcal{V}}
\newcommand{\HH}{\operatorname{H}}
\newcommand{\F}{\mathcal{F}}

\newcommand{\la}{\lambda}

\newcommand{\FF}{\mathcal{F}}

\newcommand{\VV}{\mathcal{V}}

\newcommand{\XX}{\mathcal{X}}

\renewcommand{\Gamma}{\varGamma}

\newcommand{\gln}{\mathfrak{gl}(n)}

\newtheorem{theorem}{Theorem}[subsection]

\makeatletter\let\c@fact\c@theorem\makeatother

\makeatletter\let\c@note\c@theorem\makeatother

\newtheorem{lemma}{Lemma}[subsection]
\makeatletter\let\c@lemma\c@theorem\makeatother

\makeatletter\let\c@alg\c@theorem\makeatother

\newtheorem{remark}{Remark}[subsection]
\makeatletter\let\c@remark\c@theorem\makeatother

\makeatletter\let\c@prop\c@theorem\makeatother

\makeatletter\let\c@conj\c@theorem\makeatother

\makeatletter\let\c@cor\c@theorem\makeatother

\newtheorem{defn}{Definition}[subsection]
\makeatletter\let\c@defn\c@theorem\makeatother

\theoremstyle{definition}
\newtheorem{example}{Example}[subsection]
\makeatletter\let\c@example\c@theorem\makeatother
\numberwithin{equation}{subsection}

\usepackage[capitalise]{cleveref}
\crefname{theorem}{Theorem}{Theorems}
\crefname{fact}{Fact}{Facts}
\crefname{note}{Note}{Notes}
\crefname{lemma}{Lemma}{Lemmas}
\crefname{alg}{Algorithm}{Algorithms}
\crefname{remark}{Remark}{Remarks}
\crefname{example}{Example}{Examples}
\crefname{prop}{Proposition}{Propositions}
\crefname{conj}{Conjecture}{Conjectures}
\crefname{cor}{Corollary}{Corollaries}
\crefname{defn}{Definition}{Definitions}
\crefname{equation}{\!\!}{\!\!} 

\newcounter{listequation}

\begin{document}
\title{Complexity and Support Varieties for Type P Lie Superalgebras}

\author{Brian D. Boe }
\address{Department of Mathematics \\
          University of Georgia \\
          Athens, GA 30602}
\email{brian@math.uga.edu}
\author{Jonathan R. Kujawa}
\address{Department of Mathematics \\
          University of Oklahoma \\
          Norman, OK 73019}
\thanks{Research of the second author was partially supported by NSA grant H98230-11-1-0127 and a Simons Collaboration Grant.}\
\email{kujawa@math.ou.edu}
\date{\today}
\subjclass[2010]{Primary 17B55, 17B10; Secondary 17B56}
\keywords{Lie superalgebra, representation theory, Kac module, support variety, complexity}

\begin{abstract} We compute the complexity, $z$-complexity, and support varieties of the (thick) Kac modules for the Lie superalgebras of type $P$. We also show the complexity and the $z$-complexity have geometric interpretations in terms of support and associated varieties; these results are in agreement with formulas previously discovered for other classes of Lie superalgebras. 

Our main technical tool is a recursive algorithm for constructing projective resolutions for the Kac modules. The indecomposable projective summands which appear in a given degree of the resolution are explicitly described using the combinatorics of weight diagrams. Surprisingly, the number of indecomposable summands in each degree can be computed exactly: we give an explicit formula for the corresponding generating function.
\end{abstract}

\maketitle

\section{Introduction}\label{S:intro}

\subsection{Background} Let $\fg=\fg_{\0}\oplus \fg_{\1}$ be a classical Lie superalgebra over the complex numbers. By definition $\fg_{\0}$ is a reductive Lie algebra and the adjoint action of $\fg_{\0}$ on $\fg$ is semisimple. Let $\FF$ denote the  full subcategory of all finite-dimensional $\fg$-supermodules which are completely reducible over $\fg_{\0}$ and have all weights integral.  The category $\FF$ has enough projectives and is in general not semisimple.   In \cite{BKN1, BKN2, BKN3, BKN4, BKN5} the authors and Nakano initiated a study of this category using tools imported from modular representation theory. 

Most relevant to the current paper, in \cite{BKN4} we computed the complexity, $z$-complexity, and support varieties for the Kac supermodules and simple supermodules for $\fg =\mathfrak{gl}(m|n)$.  We showed that complexity and $z$-complexity have natural interpretations in terms of the dimensions of support varieties and the associated varieties of \cite{dufloserganova}.  Entirely analogous results were subsequently obtained for Lie superalgebras in other types by El Turkey \cite{ElTurkey1,ElTurkey2}.  The main goal of the present paper is to obtain similar results for the Kac supermodules for the Lie superalgebras of type $P$, as we now describe.

\subsection{Main Results}   

We define the \emph{$z$-complexity} of a $\fg$-supermodule $M$ in $\FF$, $c_{z}(M)$, to be the rate of growth of the number of indecomposable summands  in each term of the minimal projective resolution of $M$ in $\FF$.   The \emph{complexity} of $M$,  $c_{\F}(M)$, is the rate of growth of the dimension (as a vector space) of each term of the minimal projective resolution of $M$ in $\FF$.  See \cref{SS:ratesofgrowth} for further details.   Note, $z$-complexity is invariant under category equivalences but complexity need not be preserved.

 For the remainder of the introduction let $\fg$ denote the Lie superalgebra $\fp (n)$ as defined in \cref{SS:LSAoftypeP}.  Let $X^{+}(T) = \left\{  \sum_{i=1}^{n} \mu_{i}\varepsilon_{i} \mid \mu_{i} \in \Z, \mu_{1} \geq \dotsb \geq \mu_{n}\right\}$ be the set of dominant integral weights for $\fp(n)$.  Given $\mu \in X^{+}(T)$, a \emph{run} is a maximal length sequence $\mu_{a+1}= \dotsb = \mu_{a+k}$ where $k$ is the \emph{size} of the run.   To each $\mu \in X^{+}(T)$ there is a Kac supermodule of highest weight $\mu$ which we denote by $\Delta (\mu)$.   The simple supermodules in $\FF$ are labelled by $X^{+}(T)$ and can be obtained as the unique irreducible quotients of the Kac supermodules.  Moreover, the  category $\FF$ is a highest weight category with the Kac supermodules as standard objects.   See \cref{SS:catsandmods} for details. 

Given a dominant integral weight $\mu$ for $\fp(n)$, let $o=o (\mu)$ denote the number of runs of $\mu$ which have odd size.  In \cref{T:zcomplexity,T:complexityandsupportequalities} the $z$-complexity and complexity are shown to satisfy
\begin{equation}\label{E:complexityequalities}
\begin{aligned}
c_{z} \left(\Delta(\mu) \right) &= \frac{n-o}{2}, \\ 
c_{\F}\left(\Delta(\mu) \right) &= \binom{n}{2} - \binom{o}{2}.
\end{aligned}
\end{equation}
 
Furthermore, these formulas can be interpreted in terms of varieties. By computing the support and associated varieties of $\Delta(\mu)$ we are able to verify  in \cref{T:fsupportvariety,T:complexityandsupports} that
\begin{equation}\label{E:complexityequalities2}
\begin{aligned}
c_{z} \left(\Delta(\mu) \right) &= \dim \VV_{(\ff, \ff_{\0})}\left(\Delta(\mu) \right), \\ 
c_{\F}\left(\Delta(\mu) \right) &= \dim \VV_{(\fg, \fg_{\0})}\left(\Delta(\mu) \right) + \dim \XX_{\Delta(\mu)}, 
\end{aligned}
\end{equation}
where $\VV_{(\fg , \fg_{\0})} ( \Delta(\mu) )$ (resp.\ $\VV_{(\ff , \ff_{\0})} ( \Delta(\mu) )$) is the support variety with respect to $\fg$ (resp.\ the detecting subalgebra $\ff  \subseteq \fg$), and $\XX_{\Delta(\mu)}$ is the associated variety of \cite{dufloserganova}.  Identical equalities were verified for other types in \cite{BKN4, ElTurkey1, ElTurkey2}.  Our results provide further support for the conjecture that these formulas should hold in general.  That is, for any stable classical Lie superalgebra $\fg$ over $\C$ with detecting subalgebra $\ff$ defined as in \cite{BKN1}, the formulas \cref{E:complexityequalities2} should hold when $\Delta(\mu)$ is replaced by any module $M$ in $\FF$.

\subsection{Overview} The paper is organized as follows. In \cref{S:preliminaries}, we set up the basic conventions for the Lie superalgebra $\fp (n)$, the category $\FF$, and the combinatorics of weight diagrams following \cite{WINART}.  In \cref{S:ProjResolutions} we introduce our main technical tool: an explicit projective resolution for each Kac supermodule (constructed in \cref{T:BBProjResolution}) which is described combinatorially in terms of weight diagrams.  In particular, by \cref{T:degree}  the indecomposable summands which appear in the resolution are completely described combinatorially by certain ``allowable functions'' between weight diagrams. In \cref{S:growth} weight diagrams are used to compute the rate of growth of the number of indecomposable summands for this projective resolution.  

Surprisingly, the number of indecomposable summands in this projective resolution can be computed on the nose.  If $S_{\mu}(u)$ is the generating function where the coefficient of $u^{d}$ is the number of indecomposable summands in the $d$th term of the projective resolution for $\Delta(\mu)$, then  \cref{T:zcomplexity} shows 
\[
S_{\mu}(u) = \frac{f_{\mu}(u)}{(1-u)^{\frac{n-o}{2}}} ,
\] where $f_{\mu}(u)$ is an explicit polynomial.  These calculations yield upper bounds on $z$-complexity and, when combined with dimension estimates, complexity, which match the equalities given in \cref{E:complexityequalities}.

Finally, in \cref{S:supportvarieties} we use representation theoretic computations to obtain lower bounds on the support and associated varieties for the Kac supermodules.  In combination with previous results, these lower bounds allow us to compute the support varieties for the Kac supermodules in \cref{T:complexityandsupports,T:fsupportvariety}  as well as verify  \cref{E:complexityequalities} and  \cref{E:complexityequalities2}. 

It is worth remarking these invariants do not depend on $\mu$, but only on the number of odd runs in $\mu$.  The results of \cite{BKN4} showed the degree of atypicality of the highest weight played this role for $\mathfrak{gl}(m|n)$.  This suggests that $n$ minus the number of odd runs may be the true analogue of atypicality for $\fp (n)$. 

\subsection{Additional Questions}  As the projective resolutions constructed here have the same rate of growth as the minimal projective resolutions, it is natural to ask if these are in fact minimal.  Related to this, it is worth pointing out the resolutions here fail to satisfy ``parity vanishing'' (see \cref{R:examples}). This is still true even if one takes into account parity shifts.   Since we could not find a compelling argument one way or the other on the question of minimality, we leave it as an open problem.  

The Kac supermodules, $\Delta(\mu)$, studied here are the ``thick'' Kac supermodules of \cite{WINART}.  In \emph{loc.\ cit.} they also introduce ``thin'' Kac supermodules, $\nabla (\mu)$, and show $\FF$ is a highest weight category where $\Delta(\mu)$ (resp.\ $\nabla(\mu)$) are the standard (resp.\ costandard) objects. Related to the lack of parity vanishing discussed above, it would be interesting to determine if $\FF$ admits a Kazhdan-Lusztig Theory in the sense of Cline-Parshall-Scott \cite{CPS:92,CPS:93}.  

Finally, it would be interesting to compute support varieties, complexity, and $z$-complexity for the thin Kac supermodules and the simple supermodules.  Answering these questions for $\mathfrak{gl}(m|n)$ used the existence of a duality on the category which interchanges the standard and costandard objects.  There is no such duality for $\fp(n)$, which suggests new ideas will be needed.

\subsection{A Companion Mobile App}
The first author has created a mobile app, \emph{Homologica}, to facilitate fast, easy, animated, real-time construction of the allowable functions $f:\mu\to\la$ defined in \cref{SS:KeyFunctions}. In particular, the app allows one to directly determine the indecomposable summands $P(\la)$ of a given degree in the projective resolution of the Kac module $\Delta(\mu)$ without going through the recursion of the algorithm. \emph{Homologica} is available free for iPhone and iPad on the App Store.

\section{Preliminaries} \label{S:preliminaries}  

\subsection{Superspaces} Let $\C $ be the field of complex numbers.  Unless otherwise stated, all vector spaces considered in this paper will be finite-dimensional $\C$-vector spaces.  A \emph{superspace} is a $\Z_{2}$-graded vector space, $V= V_{\0} \oplus V_{\1}$.  Given a superspace $V$ and a homogeneous element $v \in V$, we write $\p{v} \in \Z_{2}$ for its parity.  For short we call an element of $V$ \emph{even} (resp.\ \emph{odd}) if $\p{v}=\0$ (resp.\ $\p{v}=\1$).  We view $\C$ itself as a superspace concentrated in parity $\0$.  Given a superspace $V$ we say the \emph{dimension} of $V$ is $m|n$ to record that the dimension of $V_{\0}$ is $m$ and the dimension of $V_{\1}$ is $n$.  In particular, the dimension of $V$ as a vector space is $m+n$ and the \emph{superdimension} of $V$ is, by definition, $m-n$.

If $V$ and $W$ are superspaces, then $V \otimes W$ is naturally a superspace where a pure tensor has parity given by the formula $\p{v\otimes w} = \p{v} + \p{w}$ for all homogeneous $v\in V$ and $w \in W$.  Similarly, the space of $\C$-linear maps, $\Hom_{\C}(V,W)$ is naturally $\Z_{2}$-graded by declaring that a linear map $f:V \to W$ has parity $r \in \Z_{2}$ if $f(V_{s}) \subseteq W_{r+s}$ for all $s \in \Z_{2}$.

\subsection{The Lie Superalgebra of type P} \label{SS:LSAoftypeP} Let $I=I_{n|n}$ be the ordered index set consisting of the $2n$ symbols $\left\{1, \dotsc , n, 1', \dotsc , n' \right\}$.  Let $\bar{}: I \to \Z_{2}$ be the function defined by $\p{i} = \0$ if $i \in \left\{1, \dotsc ,n \right\}$ and $\p{i}=\1$ if $i \in \left\{1', \dotsc ,n' \right\}$. Let $V$ be the vector space with distinguished basis $\left\{v_{i}  \mid  i \in I \right\}$.  We define a $\Z_{2}$-grading on $V$ by declaring $\p{v}_{i}=\p{i}$ for all $i \in I$.  Let $\gl (V) = \gl(n|n)$ denote the superspace of all linear endomorphisms of $V$.  Then $\gl (V)$ is a Lie superalgebra via the graded version of the commutator bracket.  That is, 
\[
[f,g] = f \circ g - (-1)^{\p{f}\; \p{g}}g \circ f
\]
for all homogeneous $f,g \in \gl (V)$.  Note here and later we adopt the convention that a condition is only given for homogeneous elements and leave implicit the understanding that the general case can be obtained via linearity. 

Define an odd, supersymmetric, nondegenerate bilinear form on $V$ by declaring $(v_{i},v_{j'}) = (v_{j'}, v_{i}) = \delta_{i,j}$, $(v_{i},v_{j}) = (v_{i'},v_{j'}) = 0$, for $i,j=1, \dotsc ,n$.  We define a Lie superalgebra $\fg = \fp (n) \subseteq \gl (V)$ consisting of all linear maps which preserve the bilinear form for all homogeneous $x,y \in V$,
\[
\fg=\fp(n) = \left\{f \in \gl (V) \, \left| \, (f(x),y) + (-1)^{\bar{f}\;  \bar{x}}(x,f(y_{}))=0 \right. \right\}.
\]   One can easily check that the supercommutator defines a Lie superalgebra structure on $\fg=\fp (n)$.  By definition this is the Type $P$ Lie superalgebra. 

With respect to our choice of basis it is straightforward to describe $\fp (n)$ as $2n \times 2n$ complex matrices,
\begin{equation}\label{E:matrixform}
\fg = \left\{\left(\begin{matrix} A & B \\
                                  C & -A^{t}
\end{matrix} \right) \right\},
\end{equation}
where $A,B,C$ are $n \times n$ complex matrices with $B$ symmetric, $C$ skew-symmetric, and where $A^{t}$ denotes the transpose of $A$.  The $\Z_{2}$-grading is given by observing $\fg_{\0}$ is the subspace of all such matrices where $B=C=0$ and $\fg_{\1 }$ is the subspace of all such matrices where $A=0$.  In particular, notice $\fg_{\0}$ is canonically isomorphic to the Lie algebra $\mathfrak{gl}(n)$.

A $\fg$-supermodule is a superspace $M$ with an action of $\fg$ which respects the grading in the sense that $\fg_{r}.M_{s} \subseteq M_{r+s}$ for all $r,s \in \Z_{2}$, and which satisfies graded versions of the axioms for a Lie algebra module.  If $M,N$ are $\fg$-supermodules, then $\Hom_{\C}(M,N)$ inherits a $\Z_{2}$-grading as before. 
A homogeneous $\fg$-supermodule homomorphism is a homogeneous $f \in \Hom_{\C}(M,N)$ which satisfies $f(xm)=(-1)^{\p{f}\p{x}}xf(m)$ for all homogeneous $x \in \fg$ and $m \in M$.  Note we do not assume a supermodule homomorphism is homogenous; instead we will make it explicit if and when it is important that a map be homogeneous.

Let $U(\fg)$ denote the universal enveloping superalgebra of $\fg$.  Then $U(\fg )$ is a Hopf superalgebra.  In particular, if $M$ and $N$ are finite-dimensional $\fg$-supermodules (equivalently, finite-dimensional $U(\fg)$-supermodules) one can use the coproduct and antipode of $U(\fg)$ to define a $\fg$-supermodule
structure on the tensor product $M\otimes N$ and the dual $M^{*}$.  Also, viewing $\C$ as a superspace concentrated in parity zero with action given by the counit of $U(\fg )$ defines the trivial supermodule for $\fg$.  We also have $\C_{\1}$ which is $\C$ with $\fg$-supermodule structure given by the counit, but concentrated in parity $\1$.  For brevity we frequently leave the prefix ``super'' implicit in what follows.

There is a $\Z$-grading on $\fg$ which is compatible with the $\Z_{2}$ grading in that reducing modulo two recovers the $\Z_{2}$-grading.  It is given by setting $\fg_{1}$ equal to the subspace of all matrices of the form \cref{E:matrixform} where $A$ and $C$ are zero,  $\fg_{0} = \fg_{\0}$, and $\fg_{-1}$ consists of all matrices of the form \cref{E:matrixform} where $A$ and $B$ are zero.    

Let $\fb_{\0}$ denote the subalgebra of $\fg_{\0}$ consisting of matrices which are upper triangular in the $A$ block.  Then we choose the Borel subalgebra of $\fg$ to be $\fb = \fb_{\0} \oplus \fg_{-1}$.  Let $\fh$ denote the Cartan Lie subsuperalgebra of $\fg$ consisting of diagonal matrices.  We fix a basis $\varepsilon_{1}, \dotsc , \varepsilon_{n} \in \fh^{*}$ where $\varepsilon_{i}$ is the linear functional which picks out the $i$th diagonal entry of $h \in \fh$ when written as a matrix as in \cref{E:matrixform}; that is, the $i$th diagonal entry of the $A$ block. Let $X(T) = \oplus_{i=1}^{n} \Z\varepsilon_{i}$ denote the integral weight lattice.  Let 
\[
X^{+}(T) = \left\{ \left. \mu = \sum_{i=1}^{n} \mu_{i}\varepsilon_{i} \in X(T) \ \right| \,  \mu_{1} \geq \dotsb \geq \mu_{n}  \right\}.
\]

Fix $\rho = \sum_{i=1}^{n}(n-i)\varepsilon_{i}$. We denote by $\bar{\mu}$ the element $\mu+\rho = \sum_{i=1}^{n}\bar\mu_{i}\varepsilon_{i}$. It is convenient to introduce the bijection $X(T) \to \oplus_{i=1}^{n} \Z$ given by $\mu \mapsto [\mu]= [a_{1}, \dotsc , a_{n}]$ where $\mu + \rho = \sum_{i=1}^{n} a_{n-i+1}\varepsilon_{i}$.  Then $\mu$ is an element of $X^{+}(T)$ if and only if it maps to a strictly increasing sequence $[\mu]$ of integers.

\subsection{A Basis}\label{SS:conventions}

For later calculations it will be useful to fix a choice of basis for $\fp (n)$.  Given $ i,j \in I_{n|n}$, let $E_{i,j}$ denote the matrix unit with a $1$ in the $(i,j)$ position and zero everywhere else. For $1 \leq i \leq n$ let $h_{i}=E_{i,i}-E_{i',i'} \in \fh \subseteq \fg_{0}$.  For $1 \leq i,j \leq n$, $i \neq j$, let $a_{i,j}= E_{i,j}-E_{j',i'} \in \fg_{0}$.  This is a root vector for the root $\varepsilon_{i}-\varepsilon_{j}$.   For $1 \leq i \leq j \leq n$, let $b_{i,j}= E_{i,j'}+E_{j, i'}\in \fg_{1}$.  This is a root vector for the root $\varepsilon_{i}+\varepsilon_{j}$. For $1 \leq i< j \leq n$, let $c_{i,j} = E_{i',j}-E_{j',i} \in \fg_{-1}$.  This is a root vector for the root $-\varepsilon_{i}-\varepsilon_{j}$.  
We have the following commutator formulas:

\begin{equation}\label{E:commutator1}
[c_{u,v}, b_{p,q}] = \begin{cases}
                   \delta_{q,u}a_{p,v}-\delta_{q,v}a_{p,u}+\delta_{p,u}a_{q,v}-\delta_{p,v}a_{q,u}, & p < q; \\
                   \delta_{p,u}2a_{p,v} - \delta_{p,v}2a_{p,u}, & p=q.
\end{cases}
\end{equation}  When contemplating the above formulas the reader should keep in mind our conventions mean $b_{p,p}$ has a $2$ in the $(p,p')$ position.
Similarly,
\begin{equation}\label{E:commutator2}
[ a_{i,j}, c_{p,q}] = \delta_{i,p}c_{q,j} - \delta_{i,q}c_{p,j},
\end{equation}
where we adopt the convention that $c_{i,j} := - c_{j,i}$ if $i > j$ and $c_{i,i}=0$.
Further calculations show
\begin{equation}\label{E:commutator3}
[a_{i,j}, b_{p,q}] = \delta_{j,p}b_{i,q} + \delta_{j,q}b_{i,p},
\end{equation}
where we adopt the convention that $b_{i,j} :=  b_{j,i}$ if $i > j$.

\subsection{Categories and Modules}\label{SS:catsandmods}

Let $\F=\F(\fg,\fg_{\0})$ denote the full subcategory of all $\fg$-modules which are finite-dimensional, decompose into weight spaces with respect to the action of $\fh$, and all weights lie in $X(T)$. In particular, these representations are completely reducible when restricted to $\fg_{\0}$.  Moreover, this category is closed under tensor products and duals and is a monoidal supercategory.  The category $\F$ admits a parity shift functor $\Pi$ given by setting $\Pi M=\C_{\1}\otimes M$.   Except when otherwise stated, all $\fg$-modules will be assumed to be objects of $\F$.  

 For $\mu \in X^{+}(T)$, let $L_{0}(\mu)$ denote the simple $\fg_{\0} \cong \mathfrak{gl}(n)$-module of highest weight $\mu$ with respect to the Cartan and Borel subalgebras, $\fh$ and $\fb_{\0} \subseteq \fg_{\0}$, respectively.  The \emph{(thick) Kac module} for $\mu \in X^{+}(T)$ is defined to be 
\[
\Delta (\mu) = U(\fg) \otimes_{U(\fg_{\0}\oplus \fg_{-1})} L_{0}(\mu),
\] where $L_{0}(\mu)$ is viewed as a $\fg_{\0} \oplus \fg_{-1}$-module by having $\fg_{-1}$ act trivially.  By standard arguments $\Delta(\mu)$ is a highest weight module with a unique maximal proper submodule. If we write $L(\mu)$ for the simple quotient of $\Delta(\mu)$, then the set $\left\{L(\mu)  \mid \mu \in X^{+}(T) \right\}$ gives a complete set of simple modules in $\FF$ up to isomorphism and parity shift.  Let $P(\mu)$ denote the indecomposable projective cover of $L(\mu)$ and $\Delta(\mu)$.  By \cite[Proposition 2.2.2]{BKN3} it is known that the projectives and injectives in $\FF$ coincide.

\subsection{Weights and Weight Diagrams}\label{SS:WeightsandWeightDiagrams}
Given $\mu = \sum_{i=1}^{n} \mu_{i}\varepsilon_{i} \in X^{+}(T)$ define its \emph{degree of atypicality} by 
\begin{equation}\label{E:atypdef}
\atyp (\mu) = \#\left\{i \mid 1 \leq i < n, \mu_{i}=\mu_{i+1} \right\},
\end{equation} where we write $\# X$ for the cardinality of a set $X$. We call $\mu$ (and $\Delta(\mu)$ and $L(\mu)$) \emph{typical} if $\atyp (\mu)=0$ and \emph{atypical}, otherwise.  

A \emph{weight diagram with $n$ dots} is the real number line with markings on $n$ distinct integers.  We draw the markings as dots.  When needed to avoid confusion, we use tick-marks for integers which do not have dots.  We label one or more integers on the number line when needed.    When a dot in a weight diagram does not have a dot to its immediate left or right, then we call it an \emph{isolated dot}. When a dot in a weight diagram does not have a dot to its immediate left, then we call it a \emph{left-isolated} dot.  In particular, every isolated dot is left-isolated.

Given $\mu \in X^{+}(T)$, the \emph{weight diagram of $\mu$} is the  weight diagram obtained by placing a dot at the integers $a_{1}, \dots,  a_{n}$, where $\mu \mapsto [\mu]=[a_{1}, \dotsc , a_{n}]$ is the map from \cref{SS:LSAoftypeP}.  Since $\mu$ is dominant integral, the integers $a_{1}, \dotsc , a_{n}$ are strictly increasing and the result will be a weight diagram with $n$ dots.  Conversely, given a weight diagram with $n$ dots, there is a unique $\mu \in X^{+}(T)$ which corresponds to that weight diagram.  We freely identify a dominant weight with its weight diagram.

For example, if $n=4$ and $\mu = 2\varepsilon_{1} + 2\varepsilon_{2} + \varepsilon_{3} -4\varepsilon_{4}$, then $\bar\mu = \mu+\rho = 5\varepsilon_{1}+4\varepsilon_{2}+2\varepsilon_{3}  -4\varepsilon_{4}$,  $[\mu]= [-4,2,4,5]$, and the weight diagram is
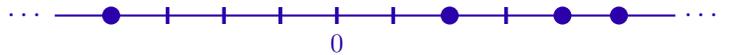
\begin{equation*}
\begin{tikzpicture}[scale=.75, color=\clr, baseline=0]
	\draw [thick ] (-5,0) to  (6,0);
	\draw (-4,0)\bdot;
	\draw (2,0)\bdot;
	\draw (4,0)\bdot;
	\draw (5,0)\bdot;
	\draw (0,0)\tick;
	\draw (-1,0)\tick;
	\draw (-2,0)\tick;
	\draw (-3,0)\tick;
	\draw (1,0)\tick;
	\draw (3,0)\tick;
        \node at (0,-0.5) {\small{$0$}};
	\node at (6.5,0) {$\cdots$};
	\node at (-5.5,0) {$\cdots$};
\end{tikzpicture} .
\end{equation*} In this example the dots at $-4$ and $2$ are isolated and the dots at $-4$, $2$, and $4$ are left-isolated.

Given $t \in \Z$ and $\mu, \la \in X^{+}(T)$, if we write $[\mu]=[a_{1}, \dotsc , a_{n}]$ and $[\la]=[b_{1}, \dotsc , b_{n}]$, then define 
\[
\ell_{t}(\lambda, \mu) = \# \left\{i \mid b_{i} \leq t  \right\}-\# \left\{i \mid a_{i} \leq t  \right\}.
\]  In terms of weight diagrams, $\ell_{t}(\lambda,\mu)$ is the difference in the number of dots which are at or to the left of the integer $t$ in the weight diagrams of $\lambda$ and $\mu$.  Since $\ell_{t}(\lambda, \mu)=0$ for all but finitely many $t$, it makes sense to define the \emph{relative length function} on $\lambda, \mu \in X^{+}(T)$ as
\begin{equation}\label{E:relativelengthdef}
\ell(\lambda, \mu) =\sum_{t=-\infty}^{\infty} \ell_{t}(\lambda, \mu).
\end{equation}
Define a partial order on $X^{+}(T)$ by declaring  $\mu \leq \lambda$ if $\lambda_{i} \leq \mu_{i}$ for $i=1, \dotsc , n$. That is, $\mu \leq \lambda$  if and only if $\ell_{t}(\lambda, \mu) \geq 0$ for all $t \in \Z$.

\section{Projective Resolutions}\label{S:ProjResolutions}  

\subsection{Translation Functors}\label{}  Regard $V$ as the natural module for $\fg$.  For each $i \in \Z$ there is an exact endofunctor $\Theta_{i}: \FF \to \FF$ given by tensoring with $V$ and projecting onto the generalized $i$-eigenspace for the action of the Casimir element. See \cite[Definition 4.1.7]{WINART} where this functor is denoted $\Theta'_{i}$. By  \cite[Theorem 7.1.1]{WINART} these functors take indecomposable projectives to indecomposable projectives or zero.   We will need the following special cases of \cite[Proposition 5.2.1, Lemmas 7.2.1 and 7.2.3]{WINART}, concerning the effect of translation functors on Kac modules and on indecomposable projective modules.  As before, we do not distinguish between a dominant weight and its weight diagram.  In the presentation of the following results we adopt the convention that the weight diagrams in question are identical except for the indicated changes.

\begin{theorem}\label{T:TranslationonStandards} Let $\lambda \in X^{+}(T)$ and let $j \in \Z$. 
\begin{enumerate}
\item  If $\lambda$ is as given below, then $\Theta_{j+1}\Delta(\lambda) \cong \Delta(\mu)$, where $\mu$ is as given below:

\begin{align*} \lambda &=
\begin{tikzpicture}[scale=.75, color=\clr, baseline=0]
	\draw [thick ] (-3,0) to  (1,0);
	\draw (-2,0)\bdot;
	\draw (-1,0)\tick;
	\draw (0,0)\tick;
        \node at (-2,-0.5) {\scriptsize{$j-1$}};
        \node at (-1,-0.5) {\scriptsize{$j$}};
        \node at (0,-0.5) {\scriptsize{$j+1$}};
	\node at (1.5,0) {$\cdots$};
	\node at (-3.5,0) {$\cdots$};
\end{tikzpicture} \\
\mu &=
\begin{tikzpicture}[scale=.75, color=\clr, baseline=0]
	\draw [thick ] (-3,0) to  (1,0);
	\draw (-1,0)\bdot;
	\draw (-2,0)\tick;
	\draw (0,0)\tick;
        \node at (-2,-0.5) {\scriptsize{$j-1$}};
        \node at (-1,-0.5) {\scriptsize{$j$}};
        \node at (0,-0.5) {\scriptsize{$j+1$}};
	\node at (1.5,0) {$\cdots$};
	\node at (-3.5,0) {$\cdots$};
\end{tikzpicture}
\end{align*}
\item  If $\lambda$ is as given below, then there is a short exact sequence 
\[
0 \to \Delta (\mu') \to \Theta_{j+1}\Delta(\lambda) \to \Delta(\mu'') \to 0
\]  where $\mu'$ and $\mu''$ are as given below:
\begin{align*} \lambda &=
\begin{tikzpicture}[scale=.75, color=\clr, baseline=0]
	\draw [thick ] (-3,0) to  (1,0);
	\draw (-2,0)\bdot;
	\draw (-1,0)\tick;
	\draw (0,0)\bdot;
        \node at (-2,-0.5) {\scriptsize{$j-1$}};
        \node at (-1,-0.5) {\scriptsize{$j$}};
        \node at (0,-0.5) {\scriptsize{$j+1$}};
	\node at (1.5,0) {$\cdots$};
	\node at (-3.5,0) {$\cdots$};
\end{tikzpicture} \\
\mu' &=
\begin{tikzpicture}[scale=.75, color=\clr, baseline=0]
	\draw [thick ] (-3,0) to  (1,0);
	\draw (-2,0)\bdot;
	\draw (-1,0)\bdot;
	\draw (0,0)\tick;
        \node at (-2,-0.5) {\scriptsize{$j-1$}};
        \node at (-1,-0.5) {\scriptsize{$j$}};
        \node at (0,-0.5) {\scriptsize{$j+1$}};
	\node at (1.5,0) {$\cdots$};
	\node at (-3.5,0) {$\cdots$};
\end{tikzpicture}\\
\mu'' &=
\begin{tikzpicture}[scale=.75, color=\clr, baseline=0]
	\draw [thick ] (-3,0) to  (1,0);
	\draw (-2,0)\tick;
	\draw (-1,0)\bdot;
	\draw (0,0)\bdot;
        \node at (-2,-0.5) {\scriptsize{$j-1$}};
        \node at (-1,-0.5) {\scriptsize{$j$}};
        \node at (0,-0.5) {\scriptsize{$j+1$}};
	\node at (1.5,0) {$\cdots$};
	\node at (-3.5,0) {$\cdots$};
\end{tikzpicture}
\end{align*}
\end{enumerate}
\end{theorem}

\begin{theorem}\label{T:TranslationonProjectives}  Let $\lambda \in X^{+}(T)$ and let $j \in \Z$.
\begin{enumerate}
\item If $\lambda$ is as below, then $\Theta_{j+1}P(\lambda) \cong P(\mu)$, where $\mu$ is as below:
\begin{align*} \lambda &=
\begin{tikzpicture}[scale=.75, color=\clr, baseline=0]
	\draw [thick ] (-3,0) to  (1,0);
	\draw (-2,0)\bdot;
	\draw (-1,0)\tick;
        \node at (-2,-0.5) {\scriptsize{$j-1$}};
        \node at (-1,-0.5) {\scriptsize{$j$}};
	\node at (1.5,0) {$\cdots$};
	\node at (-3.5,0) {$\cdots$};
\end{tikzpicture} \\
\mu &=
\begin{tikzpicture}[scale=.75, color=\clr, baseline=0]
	\draw [thick ] (-3,0) to  (1,0);
	\draw (-2,0)\tick;
	\draw (-1,0)\bdot;
        \node at (-2,-0.5) {\scriptsize{$j-1$}};
        \node at (-1,-0.5) {\scriptsize{$j$}};
	\node at (1.5,0) {$\cdots$};
	\node at (-3.5,0) {$\cdots$};
\end{tikzpicture}
\end{align*}
\item If $\lambda$ is as below, then $\Theta_{j+1}P(\lambda) = P(\mu)$, where $\mu$ is as below:
\begin{align*} \lambda &=
\begin{tikzpicture}[scale=.75, color=\clr, baseline=0]
	\draw [thick ] (-4,0) to  (1,0);
	\draw (-3,0)\tick;
	\draw (-2,0)\bdot;
	\draw (-1,0)\bdot;
        \node at (-3,-0.5) {\scriptsize{$j-2$}};
        \node at (-2,-0.5) {\scriptsize{$j-1$}};
        \node at (-1,-0.5) {\scriptsize{$j$}};
	\node at (1.5,0) {$\cdots$};
	\node at (-4.5,0) {$\cdots$};
\end{tikzpicture} \\
\mu &=
\begin{tikzpicture}[scale=.75, color=\clr, baseline=0]
	\draw [thick ] (-4,0) to  (1,0);
	\draw (-3,0)\bdot;
	\draw (-2,0)\tick;
	\draw (-1,0)\bdot;
        \node at (-3,-0.5) {\scriptsize{$j-2$}};
        \node at (-2,-0.5) {\scriptsize{$j-1$}};
        \node at (-1,-0.5) {\scriptsize{$j$}};
	\node at (1.5,0) {$\cdots$};
	\node at (-4.5,0) {$\cdots$};
\end{tikzpicture}
\end{align*}
\end{enumerate}
\end{theorem}

\subsection{The Algorithm}\label{S:TheAlgorithm}  In this subsection we describe a recursive algorithm which constructs projective resolutions for Kac modules. The algorithm is based on the one constructed by Brundan for $\gl (m|n)$ in \cite{brundan1}. First, for any $d \geq 0$ and any $\mu \in X^{+}(T)$ we explain how to construct an exact sequence of projective modules
\[
P_{d} \to P_{d-1} \to\dotsb \to P_{0} \to \Delta(\mu) \to 0.
\]  
We call such a sequence a partial projective resolution of length $d$.

The easy case is when $\mu$ is typical.  In this case $\Delta(\mu)=P(\mu)$ by \cite[Lemma 3.4.1]{WINART} and we take $P_{0}=P(\mu)$ and $P_{i}=0$ for all $i>0$.  For atypical $\mu \in X^{+}(T)$ we let $P_{0}=P(\mu)$ and let $P_{0} \to \Delta(\mu) \to 0$ be the canonical surjection.  Thus we have such a resolution for arbitrary $d$ when $\mu$ is typical and for $d=0$ when $\mu$ is atypical.  We construct longer sequences for atypical $d$ inductively as follows. Note that the terms of the sequences will be projective modules by \cite[Theorem 7.1.1]{WINART}

For the inductive step, we assume we have a partial projective resolution of length $d-1$ for all $\nu$ with $\atyp (\nu) = \atyp (\mu)$ and of length $d$ for all $\nu$ with $\atyp (\nu) < \atyp (\mu)$.  In particular, we assume $\atyp (\mu) \geq 1$ and we have a partial projective resolution of $\Delta(\mu)$ of length $d-1$.  We give a procedure for constructing a partial projective resolution of length $d$ from this data.


\noindent \textbf{Step 1:} Choose the smallest $i$ so the weight diagram for $\mu$ is of the form 
\[ \mu =
\begin{tikzpicture}[scale=.75, color=\clr, baseline=0]
	\draw [thick ] (-3,0) to  (2,0);
	\draw (-1,0)\tick;
	\draw (0,0)\bdot;
	\draw (1,0)\bdot;
        \node at (-1,-0.5) {\scriptsize{$i-1$}};
        \node at (0,-0.5) {\scriptsize{$i$}};
        \node at (1,-0.5) {\scriptsize{$i+1$}};
	\node at (2.5,0) {$\cdots$};
	\node at (-3.5,0) {$\cdots$};
\end{tikzpicture}.
\]  Such an $i$ exists since $\atyp (\mu) \geq 1$. There are now two cases.

\smallskip
\noindent \textbf{Step 2a:} Suppose $\mu$ has no dot at $i-2$. Then set $j=i$. We have

\[ \mu =
\begin{tikzpicture}[scale=.75, color=\clr, baseline=0]
	\draw [thick ] (-3,0) to  (2,0);
	\draw (-2,0)\tick;
	\draw (-1,0)\tick;
	\draw (0,0)\bdot;
	\draw (1,0)\bdot;
        \node at (-2,-0.5) {\scriptsize{$j-2$}};
        \node at (-1,-0.5) {\scriptsize{$j-1$}};
        \node at (0,-0.5) {\scriptsize{$j$}};
        \node at (1,-0.5) {\scriptsize{$j+1$}};
	\node at (2.5,0) {$\cdots$};
	\node at (-3.5,0) {$\cdots$};
\end{tikzpicture},
\] and set 
\[ \nu = 
\begin{tikzpicture}[scale=.75, color=\clr, baseline=0]
	\draw [thick ] (-3,0) to  (2,0);
	\draw (-2,0)\tick;
	\draw (-1,0)\bdot;
	\draw (0,0)\tick;
	\draw (1,0)\bdot;
        \node at (-2,-0.5) {\scriptsize{$j-2$}};
        \node at (-1,-0.5) {\scriptsize{$j-1$}};
        \node at (0,-0.5) {\scriptsize{$j$}};
        \node at (1,-0.5) {\scriptsize{$j+1$}};
	\node at (2.5,0) {$\cdots$};
	\node at (-3.5,0) {$\cdots$};
\end{tikzpicture}.
\]  Since $\atyp (\nu) < \atyp (\mu)$, $\Delta(\nu)$ has a partial projective resolution of length $d$, say 
\[
Q_{d} \to \dotsb \to Q_{1} \to Q_{0} \to \Delta(\nu) \to 0.
\]  Applying the exact functor $\Theta_{j+1}$  to this sequence yields the exact sequence 
\[
\Theta_{j+1}Q_{d} \to \dotsb \to \Theta_{j+1}Q_{1} \to \Theta_{j+1}Q_{0} \to \Theta_{j+1}\Delta (\nu) \to 0.
\] Furthermore, by  \cref{T:TranslationonStandards} there is a short exact sequence 
\[
0 \to \Delta (\mu') \to \Theta_{j+1}\Delta(\nu) \to \Delta(\mu) \to 0
\] where 
\begin{equation*} \mu' =
\begin{tikzpicture}[scale=.75, color=\clr, baseline=0]
	\draw [thick ] (-3,0) to  (2,0);
	\draw (-2,0)\tick;
	\draw (-1,0)\bdot;
	\draw (0,0)\bdot;
	\draw (1,0)\tick;
        \node at (-2,-0.5) {\scriptsize{$j-2$}};
        \node at (-1,-0.5) {\scriptsize{$j-1$}};
        \node at (0,-0.5) {\scriptsize{$j$}};
        \node at (1,-0.5) {\scriptsize{$j+1$}};
	\node at (2.5,0) {$\cdots$};
	\node at (-3.5,0) {$\cdots$};
\end{tikzpicture}.
\end{equation*}  Since $\atyp (\mu ') \leq \atyp (\mu)$, $\Delta (\mu')$ has a partial projective resolution of length $d-1$, say 
\[
U_{d-1}  \to \dotsb \to U_{1} \to U_{0} \to \Delta (\mu') \to 0.
\]  Applying the Comparison Theorem \cite[Theorem 2.2.6]{Weibel} to the inclusion $i:\Delta (\mu') \hookrightarrow \Theta_{j+1}\Delta(\nu)$ yields a double complex 

\begin{equation*}
\begin{tikzpicture}[scale=.75, color=black, baseline=0]
	\node at (-6,1) {$\cdots$};
	\node at (-5,1) {$\longrightarrow$};
	\node at (-3.5,1) {$U_{1}$};
	\node at (-2,1) {$\longrightarrow$};
	\node at (-0.5,1) {$U_{0}$};
	\node at (1,1) {$\longrightarrow$};
	\node at (2.75,1) {$\Delta(\mu')$};
	\node at (4.5,1) {$\longrightarrow$};
	\node at (5.25,1) {$0$};
	\node at (-6,-1) {$\cdots$};
	\node at (-5,-1) {$\longrightarrow$};
	\node at (-3.5,-1) {$\Theta_{j+1}Q_{1}$};
	\node at (-2,-1) {$\longrightarrow$};
	\node at (-0.5,-1) {$\Theta_{j+1}Q_{0}$};
	\node at (1,-1) {$\longrightarrow$};
	\node at (2.75,-1) {$\Theta_{j+1}\Delta(\nu)$};
	\node at (4.5,-1) {$\longrightarrow$};
	\node at (5.25,-1) {$0$};
      \draw [ thick, -> ]  (-3.5,.5) to (-3.5,-.5);
      \draw [ thick, -> ]  (-0.5,.5) to (-0.5,-.5);
      \draw [ thick, -> ]  (2.75,.5) to (2.75,-.5);
      \node at (3.1,0.1) {$i$};
\end{tikzpicture}
\end{equation*}

\noindent Taking the total complex yields an exact (by the Acyclic Assembly Lemma \cite[Lemma 2.7.3]{Weibel}) sequence 
\[
\Theta_{j+1}Q_{d} \oplus U_{d-1} \to \Theta_{j+1}Q_{d-1} \oplus U_{d-2} \to \dotsb \to \Theta_{j+1}Q_{0} \oplus \Delta (\mu') \to \Theta_{j+1}\Delta (\nu) \to 0.
\]  Factoring out $\Delta (\mu')$ from the last two terms yields a partial projective resolution of length $d$ for $\Delta (\mu)$, as desired.

\smallskip

\noindent \textbf{Step 2b:} Suppose $\mu$ has a dot at $i-2$. Then 

\[ \mu =
\begin{tikzpicture}[scale=.75, color=\clr, baseline=0]
	\draw [thick ] (-3,0) to  (2,0);
	\draw (-2,0)\bdot;
	\draw (-1,0)\tick;
	\draw (0,0)\bdot;
	\draw (1,0)\bdot;
        \node at (-2,-0.5) {\scriptsize{$i-2$}};
        \node at (-1,-0.5) {\scriptsize{$i-1$}};
        \node at (0,-0.5) {\scriptsize{$i$}};
        \node at (1,-0.5) {\scriptsize{$i+1$}};
	\node at (2.5,0) {$\cdots$};
	\node at (-3.5,0) {$\cdots$};
\end{tikzpicture},
\] and by our choice of $i$ the dots at and to the left of $i-2$ are isolated.  Choose the largest $j<i$ so that the weight diagram for $\mu$ looks locally like
\[ \mu = 
\begin{tikzpicture}[scale=.75, color=\clr, baseline=0]
	\draw [thick ] (-3,0) to  (2,0);
	\draw (-2,0)\tick;
	\draw (-1,0)\tick;
	\draw (0,0)\bdot;
	\draw (1,0)\tick;
        \node at (-2,-0.5) {\scriptsize{$j-2$}};
        \node at (-1,-0.5) {\scriptsize{$j-1$}};
        \node at (0,-0.5) {\scriptsize{$j$}};
        \node at (1,-0.5) {\scriptsize{$j+1$}};
	\node at (2.5,0) {$\cdots$};
	\node at (-3.5,0) {$\cdots$};
\end{tikzpicture}.
\] In other words, $j$ is the location of the rightmost dot at or left of $i-2$, which has no dot to its immediate right nor in the two positions to its immediate left. Set
\[ \nu = 
\begin{tikzpicture}[scale=.75, color=\clr, baseline=0]
	\draw [thick ] (-3,0) to  (2,0);
	\draw (-2,0)\tick;
	\draw (-1,0)\bdot;
	\draw (0,0)\tick;
	\draw (1,0)\tick;
        \node at (-2,-0.5) {\scriptsize{$j-2$}};
        \node at (-1,-0.5) {\scriptsize{$j-1$}};
        \node at (0,-0.5) {\scriptsize{$j$}};
        \node at (1,-0.5) {\scriptsize{$j+1$}};
	\node at (2.5,0) {$\cdots$};
	\node at (-3.5,0) {$\cdots$};
\end{tikzpicture}.
\] That is, the weight diagram of $\nu$ is obtained from that of $\mu$ by moving the dot at $j$ one position left, while leaving all other dots unchanged. Since $\atyp (\nu)=\atyp (\mu)$, then by assumption there is a partial projective resolution of length $d-1$ for $\Delta(\nu)$.  If there is a partial projective resolution of length $d$ for $\Delta(\nu)$, say 
\[
Q_{d} \to \dotsb \to Q_{0} \to \Delta(\nu) \to 0,
\]  applying $\Theta_{j+1}$ yields a partial projective resolution of length $d$ for $\Delta(\mu)$:
\[
\Theta_{j+1}Q_{d} \to \dotsb \to \Theta_{j+1}Q_{0} \to \Delta(\mu) \to 0,
\] as desired. If there is not (yet) a partial projective resolution of length $d$ for $\Delta(\nu)$, then repeat Step 2b, replacing $\mu$ with $\nu$.  After finitely many applications of Step 2b, we will have a dominant integral weight with atypicality equal to $\atyp (\mu)$ and of the form shown in Step 2a.  Applying the construction in Step 2a yields a partial projective resolution of length $d$ for that dominant integral weight and, by the repeated applications of Step 2b, it follows that $\Delta (\mu)$ has a partial projective resolution of length $d$.

Replacing $d$ by $d+1$, the same procedure constructs an exact sequence $P_{d+1} \to P_{d} \to\dotsb \to P_{0} \to \Delta(\mu) \to 0$, where we can always ensure that the terms of degree $\le d$  are \emph{the same} as the ones constructed before.    Now letting $d\to\infty$ we get a projective resolution  $P_{\bullet}(\mu) \to \Delta(\mu)$.  Summarizing, we have the following result.

\begin{theorem}\label{T:BBProjResolution}  Let $\mu \in X^{+}(T)$ be a dominant integral weight.  Then the Algorithm constructs a projective resolution of $\Delta(\mu)$:
\[
\dots \to P_{3}(\mu) \to P_{2}(\mu) \to P_{1}(\mu) \to P_{0}(\mu) \to  \Delta (\mu) \to 0.
\]  
\end{theorem}

\subsection{Key Properties} The following results summarize some key properties of the projective resolution constructed in \cref{T:BBProjResolution}.

\begin{lemma}\label{L:IsolatedDots} Suppose $P(\la)$ occurs as a summand of some $P_{d}(\mu)$, where $\mu$ has an isolated dot at position $k$. Then $\la$ has a left-isolated dot at $k$.
\end{lemma}

\begin{proof}
The proof is by induction on $\atyp(\mu)$ together with $d$. If $\mu$ is typical or $d=0$, then the only possibility is $\la=\mu$ and degree $d=0$, and the result is clear. 

So assume $\atyp(\mu) \ge 1$ and $d \ge 1$, and that the result is true for all weights of smaller atypicality in degree $d$, and for all weights of the same atypicality in degree $d-1$. In particular, the result is true for the projective indecomposables $P(\tilde\la)$ occurring in the projective resolutions of the Kac modules of highest weights $\nu$ and $\mu'$ arising in Steps 2a and 2b in the algorithm used to form $P_{d}(\mu)$. 

In Step 2a, $\mu$ does not have an isolated dot at any coordinate $j-2, \dots, j+2$, and $\mu$ differs from $\nu$ only in coordinates $j-1$ and $j$. Thus if $\mu$ has an isolated dot at $k$, so does $\nu$, and either $k<j-2$ or $k>j+2$. By induction, for every projective indecomposable summand $P(\tilde\la)$ of $P_{d}(\nu)=Q_{d}$, $\tilde\la$ has a left-isolated dot at $k$, and also at $j-1$ since $\nu$ has an isolated dot at $j-1$. So by \cref{T:TranslationonProjectives}, $\Theta_{j+1} P(\tilde\la) = P(\la)$ where $\la$ differs from $\tilde\la$ at most in positions $j-2, j-1$, and $j$. Thus $\la$ still has a left-isolated dot at $k$.

Similarly, $\mu$ differs from $\mu'$ only in coordinates $j-1, j$, and $j+1$. Thus if $\mu$ has an isolated dot at $k$, then so does $\mu'$ (with the same restrictions on $k$ as in the previous paragraph). By induction, for each indecomposable summand $P(\la)$ of $P_{d-1}(\mu')=U_{d-1}$ (which become summands of $P_{d}(\mu)$), $\la$ has a left-isolated dot at $k$.

In Step 2b, a similar analysis as for Step 2a applies to any isolated dots of $\mu$ at $k>j+1$ or $k<j-2$. We need only consider the isolated dot at $j$ in $\mu$. Note that $\nu$ has an isolated dot at $j-1$, so by induction any $\tilde\la$ for which $P(\tilde\la)$ is a summand of $P_{d}(\nu)$ has a left-isolated dot at $j-1$. By \cref{T:TranslationonProjectives}, $\Theta_{j+1}P(\tilde\la)=P(\la)$ where $\la$ has a dot at $j$ and no dot at $j-1$, as required.
\end{proof}

\begin{lemma}\label{L:IndependenceOfOrder} The result of the Algorithm is independent of the order in which the steps are applied, provided at each step the diagrams of dots and ticks are locally as specified. In other words, the phrases ``the smallest $i$'' in Step 1 and ``the largest $j$'' in Step 2b can be replaced by ``any $i$'' and ``any $j$,'' respectively.
\end{lemma}

\begin{proof} The proof is again by induction on $\atyp(\mu)$ and $d$, with induction hypothesis similar to before. Assume $\atyp(\mu) \ge 1$ and there are two allowable steps in the algorithm, either of which can be applied to $\mu$. The most complicated situation is where both steps are of type 2a, so we will treat that case, and leave to the reader the easier cases where at least one of the moves is of type 2b. So we have the following picture:

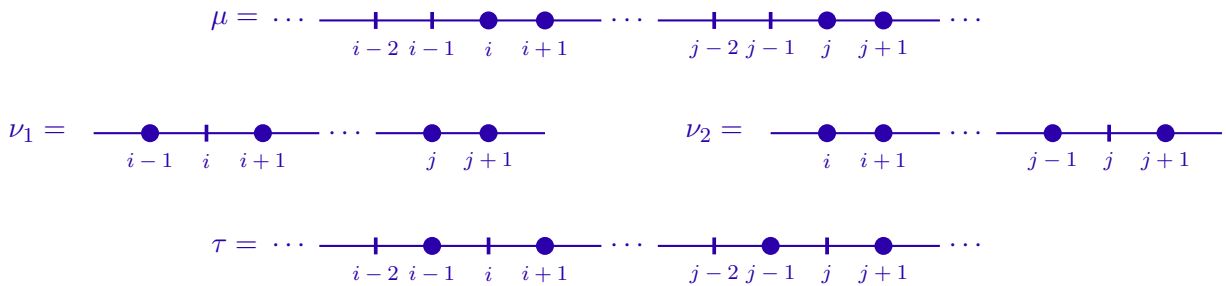
\begin{equation*}
\begin{tikzpicture}[scale=.75, color=\clr, baseline=0]
	\node at (-7.5,2) {$\mu=$};
	\draw [thick ] (-6,2) to  (-1,2);
	\node at (-6.5,2) {$\cdots$};
	\draw (-5,2)\tick;
	\draw (-4,2)\tick;
	\draw (-3,2)\bdot;
	\draw (-2,2)\bdot;
	\node at (-0.5,2) {$\cdots$};
        \node at (-5,1.5) {\scriptsize{$i-2$}};
        \node at (-4,1.5) {\scriptsize{$i-1$}};
        \node at (-3,1.5) {\scriptsize{$i$}};
        \node at (-2,1.5) {\scriptsize{$i+1$}};
	\draw [thick ] (0,2) to  (5,2);
	\draw (1,2)\tick;
	\draw (2,2)\tick;
	\draw (3,2)\bdot;
	\draw (4,2)\bdot;
	\node at (5.5,2) {$\cdots$};
        \node at (1,1.5) {\scriptsize{$j-2$}};
        \node at (2,1.5) {\scriptsize{$j-1$}};
        \node at (3,1.5) {\scriptsize{$j$}};
        \node at (4,1.5) {\scriptsize{$j+1$}};
	\node at (-11,0) {$\nu_{1}=$};
	\draw [thick ] (-10,0) to  (-6,0);
	\draw (-9,0)\bdot;
	\draw (-8,0)\tick;
	\draw (-7,0)\bdot;
        \node at (-9,-0.5) {\scriptsize{$i-1$}};
        \node at (-8,-0.5) {\scriptsize{$i$}};
        \node at (-7,-0.5) {\scriptsize{$i+1$}};
	\node at (-5.5,0) {$\cdots$};
	\draw [thick ] (-5,0) to  (-2,0);
	\draw (-4,0)\bdot;
	\draw (-3,0)\bdot;
        \node at (-4,-0.5) {\scriptsize{$j$}};
        \node at (-3,-0.5) {\scriptsize{$j+1$}};
	\node at (1,0) {$\nu_{2}=$};
	\draw [thick ] (2,0) to  (5,0);
	\draw (3,0)\bdot;
	\draw (4,0)\bdot;
        \node at (3,-0.5) {\scriptsize{$i$}};
        \node at (4,-0.5) {\scriptsize{$i+1$}};
	\node at (5.5,0) {$\cdots$};
	\draw [thick ] (6,0) to  (10,0);
	\draw (7,0)\bdot;
	\draw (8,0)\tick;
	\draw (9,0)\bdot;
        \node at (7,-0.5) {\scriptsize{$j-1$}};
        \node at (8,-0.5) {\scriptsize{$j$}};
        \node at (9,-0.5) {\scriptsize{$j+1$}};
	\node at (-7.5,-2) {$\tau=$};
	\draw [thick ] (-6,-2) to  (-1,-2);
	\node at (-6.5,-2) {$\cdots$};
	\draw (-5,-2)\tick;
	\draw (-4,-2)\bdot;
	\draw (-3,-2)\tick;
	\draw (-2,-2)\bdot;
	\node at (-0.5,-2) {$\cdots$};
        \node at (-5,-2.5) {\scriptsize{$i-2$}};
        \node at (-4,-2.5) {\scriptsize{$i-1$}};
        \node at (-3,-2.5) {\scriptsize{$i$}};
        \node at (-2,-2.5) {\scriptsize{$i+1$}};
	\draw [thick ] (0,-2) to  (5,-2);
	\draw (1,-2)\tick;
	\draw (2,-2)\bdot;
	\draw (3,-2)\tick;
	\draw (4,-2)\bdot;
	\node at (5.5,-2) {$\cdots$};
        \node at (1,-2.5) {\scriptsize{$j-2$}};
        \node at (2,-2.5) {\scriptsize{$j-1$}};
        \node at (3,-2.5) {\scriptsize{$j$}};
        \node at (4,-2.5) {\scriptsize{$j+1$}};
\end{tikzpicture}
\end{equation*}
with $i \le j-4$. There are short exact sequences
\begin{gather*}
0 \to \Delta (\mu_{1}') \to \Theta_{i+1}\Delta(\nu_{1}) \to \Delta(\mu) \to 0\\
0 \to \Delta (\mu_{2}') \to \Theta_{j+1}\Delta(\nu_{2}) \to \Delta(\mu) \to 0\\
0 \to \Delta (\nu_{1}') \to \Theta_{j+1}\Delta(\tau) \to \Delta(\nu_{1}) \to 0\\
0 \to \Delta (\nu_{2}') \to \Theta_{i+1}\Delta(\tau) \to \Delta(\nu_{2}) \to 0\\
0 \to \Delta (\sigma) \to \Theta_{i+1}\Delta(\nu_{1}') \to \Delta(\mu_{2}') \to 0\\
0 \to \Delta (\sigma) \to \Theta_{j+1}\Delta(\nu_{2}') \to \Delta(\mu_{1}') \to 0
\end{gather*}
where $\mu_{1}'$ (resp.\ $\nu_{2}'$) is obtained from $\mu$ (resp.\ $\nu_{2}$) by shifting the dots at $i, i+1$ to $i-1, i$; similarly for $\mu_{2}'$ and $\nu_{1}'$ using $j$ in place of $i$; and $\sigma$ is obtained from $\mu$ by shifting both pairs of dots $i, i+1$ and $j, j+1$ one position left. For $k=1, 2$, let us denote by $P_{d}^{k}(\mu)$ the degree $d$ term of the projective resolution for $\Delta(\mu)$ obtained via the $k^{\text{th}}$ short exact sequence above, involving $\Delta(\nu_{k})$.

Following the left path from $\tau$ via $\nu_{1}$, and using the first and third short exact sequences, we have
\begin{gather*}
P_{d}^{1}(\mu) \cong \Theta_{i+1} P_{d}(\nu_{1}) \oplus P_{d-1}(\mu_{1}')\\
P_{d}(\nu_{1}) \cong \Theta_{j+1} P_{d}(\tau) \oplus P_{d-1}(\nu_{1}').
\end{gather*}
From the second of these, and exactness of the functors $\Theta_{s}$, 
$$
\Theta_{i+1} P_{d}(\nu_{1}) \cong \Theta_{i+1} \Theta_{j+1} P_{d}(\tau) \oplus \Theta_{i+1} P_{d-1}(\nu_{1}').
$$
But using the fifth short exact sequence,
$$
P_{d-1}(\mu_{2}') \cong \Theta_{i+1} P_{d-1}(\nu_{1}') \oplus P_{d-2}(\sigma),
$$
whence
$$
\Theta_{i+1} P_{d}(\nu_{1}) \oplus P_{d-2}(\sigma) \cong \Theta_{i+1} \Theta_{j+1} P_{d}(\tau) \oplus P_{d-1}(\mu_{2}') .
$$
Therefore
$$
P_{d}^{1}(\mu) \oplus P_{d-2}(\sigma) \cong \Theta_{i+1} \Theta_{j+1} P_{d}(\tau) \oplus P_{d-1}(\mu_{2}') \oplus P_{d-1}(\mu_{1}').
$$

Similarly, following the right path from $\tau$ via $\nu_{2}$, we obtain
$$
P_{d}^{2}(\mu) \oplus P_{d-2}(\sigma) \cong \Theta_{j+1} \Theta_{i+1} P_{d}(\tau) \oplus P_{d-1}(\mu_{1}') \oplus P_{d-1}(\mu_{2}') .
$$

Since $j-i \ge 4$, \cite[Theorem 4.51]{WINART} gives that $ \Theta_{i+1} \Theta_{j+1} \cong \Theta_{j+1} \Theta_{i+1}$. Hence
$
P_{d}^{1}(\mu) \cong P_{d}^{2}(\mu)
$
and this case of the independence of path is proved.
\end{proof}

\begin{remark}
Not only is the resolution constructed by the Algorithm unique, there is a ``canonical'' typical weight $\mu_{0}$ to which $\mu$ will be reduced by the Algorithm, independent of the order in which the steps are applied. Roughly speaking, $\mu_{0}$ is obtained by shifting adjacent dots left so there is one tick between them, and, recursively, moving left any other dot when a dot needs to move into the spot to its immediate right, so as never to create any new pairs of adjacent dots.
\end{remark}

\subsection{Allowable Functions}\label{SS:KeyFunctions}

Given $\mu, \la \in X^{+}(T)$ with $[\mu]=[a_{1}, \dotsc , a_{n}]$ and $[\la]=[b_{1}, \dotsc , b_{n}]$, we say a function is of \emph{type} $\mu \to \lambda$ and write $f: \mu \to \lambda$ if $f: \left\{a_{1}, \dotsc , a_{n} \right\} \to \left\{b_{1}, \dotsc , b_{n} \right\}$ is a bijection.  It is convenient to draw a function of type $\mu \to \lambda$ using weight diagrams as follows:

\pagebreak

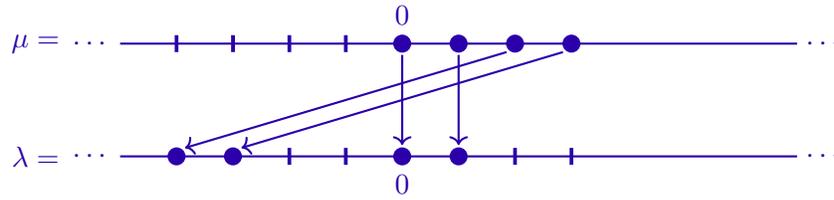
\begin{equation*}
\begin{tikzpicture}[scale=.75, color=\clr, baseline=0]
	\node at (-6.5,1) {$\mu=$};
	\draw [thick ] (-5,1) to  (7,1);
	\draw (-4,1)\tick;
	\draw (-3,1)\tick;
	\draw (-2,1)\tick;
	\draw (-1,1)\tick;
	\draw (0,1)\bdot;
	\draw (1,1)\bdot;
	\draw (2,1)\bdot;
	\draw (3,1)\bdot;
        \node at (0,1.5) {$0$};
	\node at (7.5,1) {$\cdots$};
	\node at (-5.5,1) {$\cdots$};
	\node at (-6.5,-1) {$\lambda=$};
	\draw [thick ] (-5,-1) to  (7,-1);
	\draw (-4,-1)\bdot;
	\draw (-3,-1)\bdot;
	\draw (-2,-1)\tick;
	\draw (-1,-1)\tick;
	\draw (0,-1)\bdot;
	\draw (1,-1)\bdot;
	\draw (2,-1)\tick;
	\draw (3,-1)\tick;
        \node at (0,-1.5) {$0$};
	\node at (7.5,-1) {$\cdots$};
	\node at (-5.5,-1) {$\cdots$};
      \draw [ thick, -> ]  (0,.8) to (0,-.8);
      \draw [ thick,-> ]  (1,.8) to (1,-.8);
      \draw [ thick, ->, ]  (1.85,0.85) to  (-3.85,-.85); 
      \draw [ thick, ->, ]  (2.85,0.85) to  (-2.85,-.85); 
\ .
\end{tikzpicture}
\end{equation*}  As the reader may have guessed, this picture depicts the function $f(0)=0$, $f(1)=1$, $f(2)=-4$, and $f(3)=-3$.

Given a function $\tilde{f}: \tilde{\mu} \to \tilde{\la}$ there are three distinguished ``moves'' which construct a new function.  These moves are local in the sense that there may be dots and arrows other than those depicted, but they are assumed to be left unchanged by the move.

\smallskip
\noindent \textbf{Move 1 (Sliding Isolated Dots):} Say  $\tilde{f}: \tilde{\mu} \to \tilde{\lambda}$ is as follows:
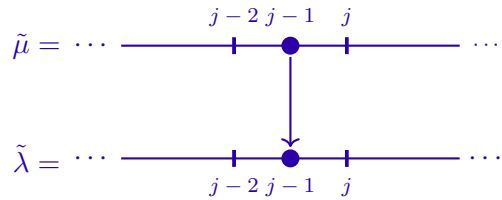
\begin{equation*}
\begin{tikzpicture}[scale=.75, color=\clr, baseline=0]
	\draw [thick ] (-3,1) to  (3,1);
	\node at (-4.5,1) {$\tilde{\mu}=$};
	\draw (-1,1)\tick;
	\draw (0,1)\bdot;
	\draw (1,1)\tick;
        \node at (-1,1.5) {\scriptsize{$j-2$}};
        \node at (0,1.5) {\scriptsize{$j-1$}};
        \node at (1,1.5) {\scriptsize{$j$}};
	\node at (3.5,1) {\scriptsize{$\cdots$}};
	\node at (-3.5,1) {$\cdots$};
	\draw [thick ] (-3,-1) to  (3,-1);
	\node at (-4.5,-1) {$\tilde{\lambda}=$};
	\draw (-1,-1)\tick;
	\draw (0,-1)\bdot;
	\draw (1,-1)\tick;
        \node at (-1,-1.5) {\scriptsize{$j-2$}};
        \node at (0,-1.5) {\scriptsize{$j-1$}};
        \node at (1,-1.5) {\scriptsize{$j$}};
	\node at (3.5,-1) {$\cdots$};
	\node at (-3.5,-1) {$\cdots$};
      \draw [ thick, -> ]  (0,.8) to (0,-.8);
\end{tikzpicture}.
\end{equation*} Then by definition Move 1 yields the function $f: \mu \to \lambda$:
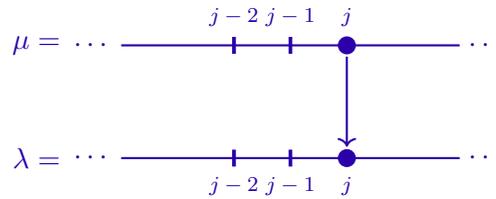
\begin{equation*}
\begin{tikzpicture}[scale=.75, color=\clr, baseline=0]
	\draw [thick ] (-3,1) to  (3,1);
	\node at (-4.5,1) {$\mu=$};
	\draw (-1,1)\tick;
	\draw (0,1)\tick;
	\draw (1,1)\bdot;
        \node at (-1,1.5) {\scriptsize{$j-2$}};
        \node at (0,1.5) {\scriptsize{$j-1$}};
        \node at (1,1.5) {\scriptsize{$j$}};
	\node at (3.5,1) {$\cdots$};
	\node at (-3.5,1) {$\cdots$};
	\draw [thick ] (-3,-1) to  (3,-1);
	\node at (-4.5,-1) {$\lambda=$};
	\draw (-1,-1)\tick;
	\draw (0,-1)\tick;
	\draw (1,-1)\bdot;
        \node at (-1,-1.5) {\scriptsize{$j-2$}};
        \node at (0,-1.5) {\scriptsize{$j-1$}};
        \node at (1,-1.5) {\scriptsize{$j$}};
	\node at (3.5,-1) {$\cdots$};
	\node at (-3.5,-1) {$\cdots$};
      \draw [ thick, -> ]  (1,.8) to (1,-.8);
\end{tikzpicture}.
\end{equation*}

\smallskip
\noindent \textbf{Move 2 (Leapfrogging):} Say  $\tilde{f}: \tilde{\mu} \to \tilde{\lambda}$ is as follows:
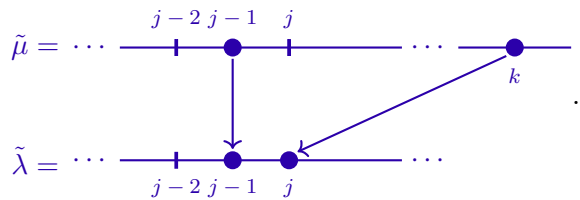
\begin{equation*}
\begin{tikzpicture}[scale=.75, color=\clr, baseline=0]
	\draw [thick ] (-3,1) to  (2,1);
	\node at (-4.5,1) {$\tilde{\mu}=$};
	\draw (-2,1)\tick;
	\draw (-1,1)\bdot;
	\draw (0,1)\tick;
        \node at (-2,1.5) {\scriptsize{$j-2$}};
        \node at (-1,1.5) {\scriptsize{$j-1$}};
        \node at (0,1.5) {\scriptsize{$j$}};
	\node at (2.5,1) {$\cdots$};
	\node at (-3.5,1) {$\cdots$};
	\draw [thick ] (3,1) to  (5,1);
	\draw (4,1)\bdot;
        \node at (4,.5) {\scriptsize{$k$}};
	\draw [thick ] (-3,-1) to  (2,-1);
	\node at (-4.5,-1) {$\tilde{\lambda}=$};
	\draw (-2,-1)\tick;
	\draw (-1,-1)\bdot;
	\draw (0,-1)\bdot;
        \node at (-2,-1.5) {\scriptsize{$j-2$}};
        \node at (-1,-1.5) {\scriptsize{$j-1$}};
        \node at (0,-1.5) {\scriptsize{$j$}};
	\node at (2.5,-1) {$\cdots$};
	\node at (-3.5,-1) {$\cdots$};
      \draw [ thick, -> ]  (-1,.8) to (-1,-.8);
      \draw [ thick, -> ]  (3.85,0.85) to  (0.15,-.85); 
\end{tikzpicture}.
\end{equation*} Then by definition Move 2 yields the function $f: \mu \to \lambda$:
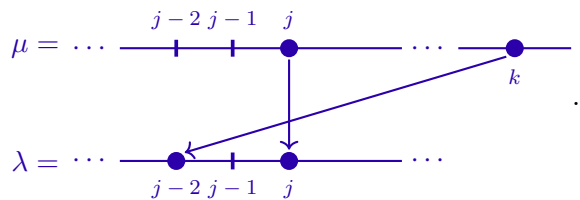
\begin{equation*}
\begin{tikzpicture}[scale=.75, color=\clr, baseline=0]
	\draw [thick ] (-3,1) to  (2,1);
	\node at (-4.5,1) {$\mu=$};
	\draw (-2,1)\tick;
	\draw (-1,1)\tick;
	\draw (0,1)\bdot;
        \node at (-2,1.5) {\scriptsize{$j-2$}};
        \node at (-1,1.5) {\scriptsize{$j-1$}};
        \node at (0,1.5) {\scriptsize{$j$}};
	\node at (2.5,1) {$\cdots$};
	\node at (-3.5,1) {$\cdots$};
	\draw [thick ] (3,1) to  (5,1);
	\draw (4,1)\bdot;
        \node at (4,.5) {\scriptsize{$k$}};
	\draw [thick ] (-3,-1) to  (2,-1);
	\node at (-4.5,-1) {$\lambda=$};
	\draw (-2,-1)\bdot;
	\draw (-1,-1)\tick;
	\draw (0,-1)\bdot;
        \node at (-2,-1.5) {\scriptsize{$j-2$}};
        \node at (-1,-1.5) {\scriptsize{$j-1$}};
        \node at (0,-1.5) {\scriptsize{$j$}};
	\node at (2.5,-1) {$\cdots$};
	\node at (-3.5,-1) {$\cdots$};
      \draw [ thick, -> ]  (0,.8) to (0,-.8);
      \draw [ thick, -> ]  (3.85,0.85) to  (-1.85,-.85); 
\end{tikzpicture}.
\end{equation*}

\smallskip
\pagebreak
\noindent \textbf{Move 3 (Sliding an Isolated Pair):} Say  $\tilde{f}: \tilde{\mu} \to \tilde{\lambda}$ is as follows:

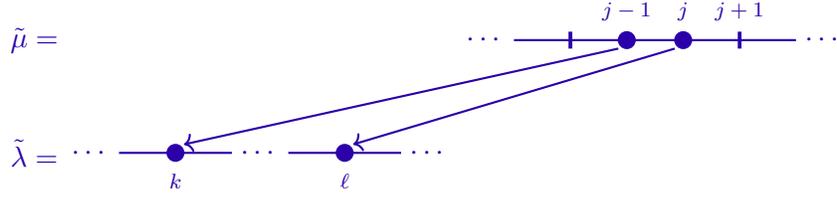
\begin{equation*}
\begin{tikzpicture}[scale=.75, color=\clr, baseline=0]
	\draw [thick ] (3,1) to  (8,1);
	\node at (-5.5,1) {$\tilde{\mu}=$};
	\draw (4,1)\tick;
	\draw (5,1)\bdot;
	\draw (6,1)\bdot;
	\draw (7,1)\tick;
        \node at (5,1.5) {\scriptsize{$j-1$}};
        \node at (6,1.5) {\scriptsize{$j$}};
        \node at (7,1.5) {\scriptsize{$j+1$}};
	\node at (2.5,1) {$\cdots$};
	\node at (8.5,1) {$\cdots$};
	\draw [thick ] (-4,-1) to  (-2,-1);
	\draw [thick ] (-1,-1) to  (1,-1);
	\node at (-5.5,-1) {$\tilde{\lambda}=$};
	\draw (-3,-1)\bdot;
	\draw (0,-1)\bdot;
        \node at (-3,-1.5) {\scriptsize{$k$}};
        \node at (0,-1.5) {\scriptsize{$\ell$}};
	\node at (1.5,-1) {$\cdots$};
	\node at (-1.5,-1) {$\cdots$};
	\node at (-4.5,-1) {$\cdots$};
      \draw [ thick, -> ]  (4.85,0.85) to  (-2.85,-.85); 
      \draw [ thick, -> ]  (5.85,0.85) to (0.15,-.85); 
\end{tikzpicture}.
\end{equation*}
Then by definition Move 3 yields the function $f: \mu \to \lambda$:
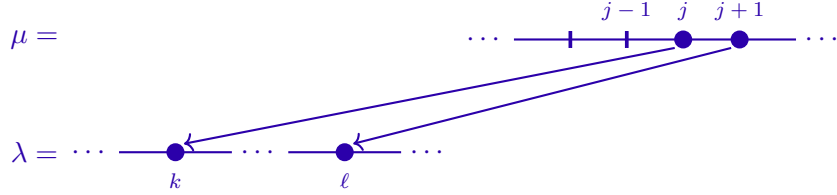
\begin{equation*}
\begin{tikzpicture}[scale=.75, color=\clr, baseline=0]
	\draw [thick ] (3,1) to  (8,1);
	\node at (-5.5,1) {$\mu=$};
	\draw (4,1)\tick;
	\draw (5,1)\tick;
	\draw (6,1)\bdot;
	\draw (7,1)\bdot;
        \node at (5,1.5) {\scriptsize{$j-1$}};
        \node at (6,1.5) {\scriptsize{$j$}};
        \node at (7,1.5) {\scriptsize{$j+1$}};
	\node at (2.5,1) {$\cdots$};
	\node at (8.5,1) {$\cdots$};
	\draw [thick ] (-4,-1) to  (-2,-1);
	\draw [thick ] (-1,-1) to  (1,-1);
	\node at (-5.5,-1) {$\lambda=$};
	\draw (-3,-1)\bdot;
	\draw (0,-1)\bdot;
        \node at (-3,-1.5) {\scriptsize{$k$}};
        \node at (0,-1.5) {\scriptsize{$\ell$}};
	\node at (1.5,-1) {$\cdots$};
	\node at (-1.5,-1) {$\cdots$};
	\node at (-4.5,-1) {$\cdots$};
      \draw [ thick, -> ]  (5.85,0.85) to (-2.85,-.85); 
      \draw [ thick, -> ]  (6.85,0.85) to  (0.15,-.85); 
\end{tikzpicture}.
\end{equation*}

Given dominant integral weights $\mu, \lambda \in X^{+}(T)$, we call a function $f: \mu \to \lambda$ an \emph{allowable function of type $\mu \to \lambda$} if $f$ can be obtained from the identity function $\operatorname{Id}_{\gamma}: \gamma \to \gamma$ for some typical dominant integral weight $\gamma \in X^{+}(T)$ using a finite sequence of Moves 1, 2, and 3. It is clear inductively from the description of the Moves that an allowable function is nonincreasing; in other words, if $f: \mu \to \lambda$ is an allowable function, then $\mu \le \la$ in the partial order of \cref{SS:WeightsandWeightDiagrams}.

\subsection{Leapfrogging}  We now explain how the combinatorics of allowable functions describes which projective indecomposables appear in the resolution of $\Delta(\mu)$ and in which degrees they appear.

\begin{theorem}\label{T:Projectivesandfunctions}  Let $\mu \in X^{+}(T)$ be a dominant integral weight.  Then $P(\lambda)$ appears in the projective resolution for $\Delta(\mu)$ constructed in \cref{T:BBProjResolution} if and only if there is an allowable function $f: \mu \to \lambda$. In particular, $\mu\le\la$ in the partial order of \cref{SS:WeightsandWeightDiagrams}.
\end{theorem}

\begin{proof}  This is a combinatorial reformulation of the algorithm. The result is clearly true when $\mu$ is typical. So assume $\atyp (\mu) \geq 1$, and that the result is true inductively as in \cref{S:TheAlgorithm}. Let us first prove the ``only if'' assertion.

First, suppose that $P_{d}(\mu)$ is obtained as in Step 2b of the algorithm. Then $P_{d} = \Theta_{j+1} Q_{d}$. So $P(\la) = \Theta_{j+1} P(\tilde\la)$ for some summand $P(\tilde\la)$ of $Q_{d}$ for $\Delta(\nu)$. By induction, there is an allowable function $\tilde f : \nu \to \tilde\la$. Because $j-1$ is an isolated dot in $\nu$, it follows from the description of the Moves, and by \cref{L:IsolatedDots}, that $\tilde f(j-1) = j-1$ and $\tilde\la$ has no dot at $j-2$. There are two cases, according to whether or not $\tilde\la$ has a dot at $j$. 

If $\tilde\la$ does not have a dot at $j$, then we are in the setting of Move 1. According to \cref{T:TranslationonProjectives} (1), $\Theta_{j+1} P(\tilde\la) = P(\la)$, where $\la$ is obtained from $\tilde\la$ by moving the dot from $j-1$ to $j$. Note that $\mu$ is obtained from $\nu$ by exactly the same procedure. Thus we construct $f$ from $\tilde f$ via Move 1.

On the other hand, if $\tilde\la$ does have a dot at $j$, then we are in the setting of Move 2. (Notice that since $\nu$ does not have a dot at $j$, and since every allowable function is nonincreasing, there must exist $k>j$ with $\tilde f(k)=j$ as in the Move 2 diagram.) By \cref{T:TranslationonProjectives} (2), $\Theta_{j+1} P(\tilde\la) = P(\la)$, where $\la$ is obtained from $\tilde\la$ by moving the dot from $j-1$ to $j-2$. On the other hand, $\mu$ is still obtained from $\nu$ by moving the dot from $j-1$ to $j$. Thus we construct $f$ from $\tilde f$ via Move 2.

Second, suppose that $P_{d}(\mu)$ is obtained as in Step 2a of the algorithm. Then $P_{d} = \Theta_{j+1} Q_{d} \oplus U_{d-1}$. The analysis of the summands coming from $\Theta_{j+1} Q_{d}$ is exactly the same as in the first case, since $\nu$ again has an isolated dot at $j-1$. The summands coming from $U_{d-1}$ are as pictured in Move 3, since $\mu$ is obtained from $\mu'$ by sliding an isolated pair of dots at $j-1$ and $j$ one step right (to $j$ and $j+1$), whereas the indecomposable projective $P(\la)$ does not change. Thus we obtain $f$ from $\tilde f$ via Move 3.

The converse follows largely by the same line of argument, using induction, the fact that it is trivially true when $\mu$ is typical, and the fact that Moves 1, 2, and 3 correspond precisely to the way that summands $P(\la)$ of $P_{d}(\mu)$ arise in the algorithm. There is one subtle point, however, and that is that the algorithm steps are to be carried out in a very specific order, whereas an allowable function could be constructed via a sequence of Moves 1, 2, and 3 in \emph{any} order. The fact that the order does not matter follows from \cref{L:IndependenceOfOrder}.

The last statement of the theorem follows from the observation at the end of \cref{SS:KeyFunctions}.
\end{proof}

\begin{defn}\label{D:leapfrogging} Let $\mu, \lambda \in X^{+}(T)$ with $[\mu]=[a_{1},\dotsc ,a_{n}]$.  Let $f:\mu \to \lambda$ be a function of type $\mu \to \lambda$.  If $1 \leq i < j \leq n$, but $f(a_{i}) > f(a_{j})$, then we call the pair $(i,j)$ a \emph{leapfrogging pair for $f$}. 

Given $f: \mu \to \lambda$, let $\mathtt{L}(f: \mu \to \lambda)$
be the total number of leapfrogging pairs for $f: \mu \to \lambda$.  That is, 
\begin{equation*}
\mathtt{L}(f) =\mathtt{L}(f: \mu \to \lambda)=\sum_{j=2}^{n} \# \left\{ i < j \mid f(a_{i}) > f(a_{j}) \right\}.
\end{equation*}
\end{defn}
\noindent In terms of the pictorial representation of $f:\mu \to \lambda$, if the arrows are drawn as straight lines so that at most two lines intersect at each point, then $\mathtt{L}(f)$ is simply the number of crossings.  For example, if $f$ is the function drawn at the beginning of \cref{SS:KeyFunctions}, then $\mathtt{L}(f)=4$.

Given $\mu,\lambda \in X^{+}(T)$ and an allowable $f: \mu \to \lambda$, the previous result shows $P(\lambda)$ appears as a direct summand of some term in the projective resolution of $\Delta(\mu)$ constructed in \cref{T:BBProjResolution}.  The next theorem sharpens this result.

\begin{theorem}\label{T:degree}  Given $\mu,\lambda \in X^{+}(T)$, then $P(\lambda)$ occurs as a direct summand of $P_{d}(\mu)$ if and only if there exists an allowable function $f: \mu \to \lambda$ such that 
\[
d=  \tfrac{1}{2}\ell(\lambda, \mu) - \mathtt{L}(f),
\]  where $\ell$ is the relative length function from \cref{E:relativelengthdef}. Moreover, the number of times $P(\lambda)$ appears as a direct summand of $P_{d}(\mu)$ equals the number of such allowable functions.
\end{theorem}

\begin{proof} \cref{T:Projectivesandfunctions} shows there is a bijection between functions $f$ of type $\mu\to\la$ and summands $P(\la)$ (counted with multiplicities) in $P_{\bullet}(\mu)$. It remains to check the degree formula. This is proved by induction much as in other proofs, with the base case $f=\operatorname{Id}_{\mu}:\mu\to\mu$ correctly giving degree $d=0$.

For the inductive step, we analyze projective indecomposable summands $P(\la)$ arising as in the proof of \cref{T:Projectivesandfunctions} via each of Moves 1, 2, and 3 in turn,  assuming inductively that $P(\tilde\la)$ occurs in degree $\tilde d$ of $P_{\bullet}(\tilde\mu)$, where $\tilde d$ is given by the claimed formula using $\tilde f$. 

In the case of Move 1, we have $\sum_{t}\mu_{t}=1+\sum_{t}\tilde\mu_{t}$ and $\sum_{t}\la_{t}=1+\sum_{t}\tilde\la_{t}$, so $\ell(\la,\mu)=\ell(\tilde\la,\tilde\mu)$. Also $\mathtt{L}(f)=\mathtt{L}(\tilde f)$. So the claimed formula gives $d=\tilde d$ as desired: recall from the proof of \cref{T:Projectivesandfunctions} that Move 1 corresponds to summands coming from $\Theta_{j+1}Q_{d}\hookrightarrow P_{d}$ in the Algorithm.

In the case of Move 2, $\sum_{t}\mu_{t}=1+\sum_{t}\tilde\mu_{t}$ whereas $\sum_{t}\la_{t}=-1+\sum_{t}\tilde\la_{t}$, so $\ell(\la,\mu)=\ell(\tilde\la,\tilde\mu)+2$. On the other hand $\mathtt{L}(f)=1+\mathtt{L}(\tilde f)$. So the claimed formula gives $d=\tilde d$ as desired: Move 2 also corresponds to  summands coming from $\Theta_{j+1}Q_{d}\hookrightarrow P_{d}$ in the Algorithm.

Finally for Move 3, $\sum_{t}\mu_{t}=2+\sum_{t}\tilde\mu_{t}$ and $\sum_{t}\la_{t}=\sum_{t}\tilde\la_{t}$, so $\ell(\la,\mu)=\ell(\tilde\la,\tilde\mu)+2$. And $\mathtt{L}(f)=\mathtt{L}(\tilde f)$. So the claimed formula gives $d=\tilde d + 1$ as desired: Move 3  corresponds to summands coming from $U_{d-1}\hookrightarrow P_{d}$ in Step 2a of the Algorithm.
\end{proof}

\begin{example}\label{R:examples}
(1)  It can happen that the same indecomposable projective $P(\la)$ occurs in two successive degrees $P_{d}$ and $P_{d+1}$. Thus the resolution $P_{\bullet}(\mu)$ does not satisfy the ``parity vanishing'' condition expected if it were a minimal resolution and if the highest weight category $\F(\fg,\fg_{\0})$ were to have a Kazhdan-Lusztig Theory in the sense of Cline-Parshall-Scott \cite{CPS:92,CPS:93}. A small example is $\mu=[3,4,5,7,8],\ \la=[0,1,3,5,6]$. Two allowable functions $f:\mu\to\la$ are pictured here:
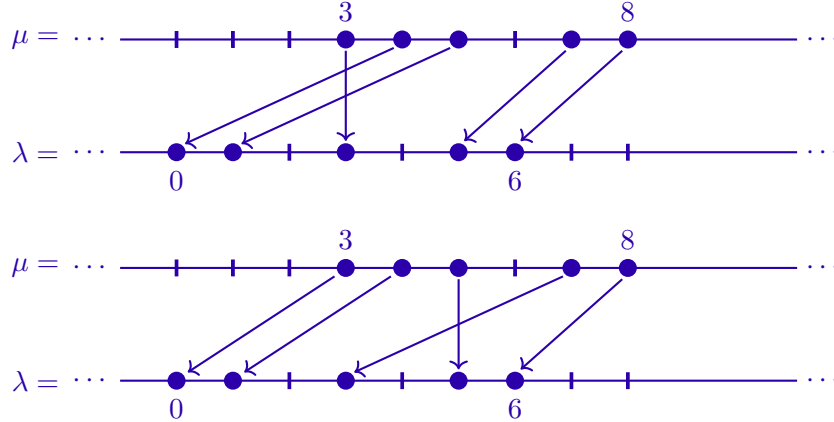
\begin{equation*}
\begin{tikzpicture}[scale=.75, color=\clr, baseline=0]
	\node at (-6.5,1) {$\mu=$};
	\draw [thick ] (-5,1) to  (7,1);
	\draw (-4,1)\tick;
	\draw (-3,1)\tick;
	\draw (-2,1)\tick;
	\draw (-1,1)\bdot;
	\draw (0,1)\bdot;
	\draw (1,1)\bdot;
	\draw (2,1)\tick;
	\draw (3,1)\bdot;
	\draw (4,1)\bdot;
        \node at (-1,1.5) {$3$};
        \node at (4,1.5) {$8$};
	\node at (7.5,1) {$\cdots$};
	\node at (-5.5,1) {$\cdots$};
	\node at (-6.5,-1) {$\lambda=$};
	\draw [thick ] (-5,-1) to  (7,-1);
	\draw (-4,-1)\bdot;
	\draw (-3,-1)\bdot;
	\draw (-2,-1)\tick;
	\draw (-1,-1)\bdot;
	\draw (0,-1)\tick;
	\draw (1,-1)\bdot;
	\draw (2,-1)\bdot;
	\draw (3,-1)\tick;
	\draw (4,-1)\tick;
        \node at (-4,-1.5) {$0$};
        \node at (2,-1.5) {$6$};
	\node at (7.5,-1) {$\cdots$};
	\node at (-5.5,-1) {$\cdots$};
      \draw [ thick, -> ]  (-1,.8) to (-1,-.8);
      \draw [ thick, -> ]  (-0.15,0.85) to (-3.85,-.85); 
      \draw [ thick, -> ]  (0.85,0.85) to  (-2.85,-.85); 
      \draw [ thick, -> ]  (2.9,0.8) to  (1.1,-.8);
      \draw [ thick, -> ]  (3.9,0.8) to  (2.1,-.8);
\end{tikzpicture}
\end{equation*}
\begin{equation*}
\begin{tikzpicture}[scale=.75, color=\clr, baseline=0]
	\node at (-6.5,1) {$\mu=$};
	\draw [thick ] (-5,1) to  (7,1);
	\draw (-4,1)\tick;
	\draw (-3,1)\tick;
	\draw (-2,1)\tick;
	\draw (-1,1)\bdot;
	\draw (0,1)\bdot;
	\draw (1,1)\bdot;
	\draw (2,1)\tick;
	\draw (3,1)\bdot;
	\draw (4,1)\bdot;
        \node at (-1,1.5) {$3$};
        \node at (4,1.5) {$8$};
	\node at (7.5,1) {$\cdots$};
	\node at (-5.5,1) {$\cdots$};
	\node at (-6.5,-1) {$\lambda=$};
	\draw [thick ] (-5,-1) to  (7,-1);
	\draw (-4,-1)\bdot;
	\draw (-3,-1)\bdot;
	\draw (-2,-1)\tick;
	\draw (-1,-1)\bdot;
	\draw (0,-1)\tick;
	\draw (1,-1)\bdot;
	\draw (2,-1)\bdot;
	\draw (3,-1)\tick;
	\draw (4,-1)\tick;
        \node at (-4,-1.5) {$0$};
        \node at (2,-1.5) {$6$};
	\node at (7.5,-1) {$\cdots$};
	\node at (-5.5,-1) {$\cdots$};
      \draw [ thick, -> ]  (-1.2,0.85) to  (-3.8,-.85);
      \draw [ thick, -> ]  (-0.2,0.85) to (-2.8,-.85);
      \draw [ thick, -> ]  (1,.8) to (1,-.8);
      \draw [ thick, -> ]  (2.85,0.85) to  (-0.85,-.85); 
      \draw [ thick, -> ]  (3.9,0.8) to  (2.1,-.8);
\end{tikzpicture}
\end{equation*}
We have $\ell(\la,\mu)=12$. The first function has $\mathtt{L}(f)=2$ so $d=4$, whereas the second function has $\mathtt{L}(f)=1$ so $d=5$.

(2)  It is also possible to have some $P(\la)$ appear more than once as a summand in a given $P_{d}(\mu)$. For example, with $\mu=[0,1,2,3,8,9,10,11],\ \la=[-4,-3,0,1,4,5,8,9]$ we have two allowable functions $f:\mu\to\la$ pictured here: 
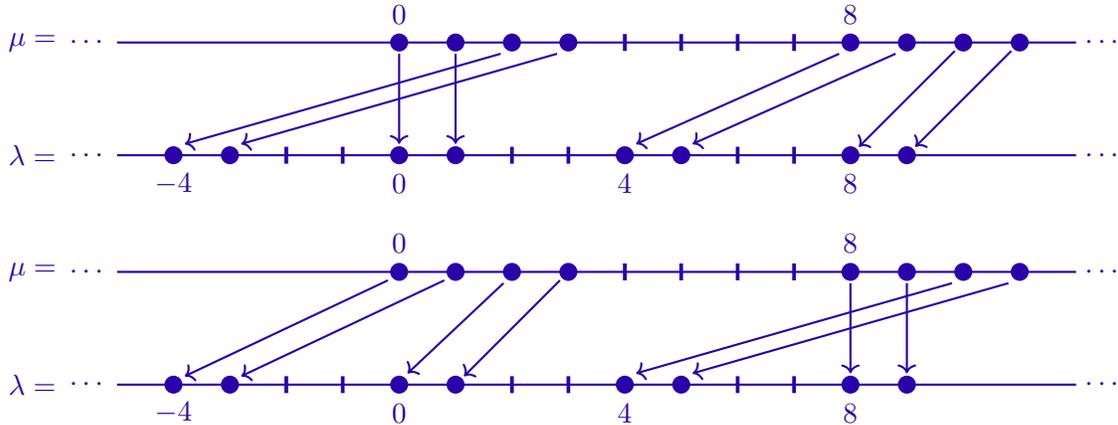
\begin{equation*}
\begin{tikzpicture}[scale=.75, color=\clr, baseline=0]
	\node at (-6.5,1) {$\mu=$};
	\draw [thick ] (-5,1) to  (12,1);
	\draw (0,1)\bdot;
	\draw (1,1)\bdot;
	\draw (2,1)\bdot;
	\draw (3,1)\bdot;
	\draw (4,1)\tick;
	\draw (5,1)\tick;
	\draw (6,1)\tick;
	\draw (7,1)\tick;
	\draw (8,1)\bdot;
	\draw (9,1)\bdot;
	\draw (10,1)\bdot;
	\draw (11,1)\bdot;
        \node at (0,1.5) {$0$};
        \node at (8,1.5) {$8$};
	\node at (12.5,1) {$\cdots$};
	\node at (-5.5,1) {$\cdots$};
	\node at (-6.5,-1) {$\lambda=$};
	\draw [thick ] (-5,-1) to  (12,-1);
	\draw (-4,-1)\bdot;
	\draw (-3,-1)\bdot;
	\draw (-2,-1)\tick;
	\draw (-1,-1)\tick;
	\draw (0,-1)\bdot;
	\draw (1,-1)\bdot;
	\draw (2,-1)\tick;
	\draw (3,-1)\tick;
	\draw (4,-1)\bdot;
	\draw (5,-1)\bdot;
	\draw (6,-1)\tick;
	\draw (7,-1)\tick;
	\draw (8,-1)\bdot;
	\draw (9,-1)\bdot;
        \node at (-4,-1.5) {$-4$};
        \node at (0,-1.5) {$0$};
        \node at (4,-1.5) {$4$};
        \node at (8,-1.5) {$8$};
	\node at (12.5,-1) {$\cdots$};
	\node at (-5.5,-1) {$\cdots$};
      \draw [ thick, -> ]  (0,0.8) to (0,-.8);
      \draw [ thick, -> ]  (1,0.8) to (1,-.8);
      \draw [ thick, -> ]  (1.8,.8) to (-3.8,-.8);
      \draw [ thick, -> ]  (2.8,0.8) to (-2.8,-.8);
      \draw [ thick, -> ]  (7.8,0.8) to (4.2,-.8);
      \draw [ thick, -> ]  (8.8,0.8) to (5.2,-.8);
      \draw [thick, -> ] (9.85,.85) to (8.15,-.85);
      \draw [thick, -> ] (10.85,.85) to (9.15,-.85);
\end{tikzpicture}
\end{equation*}
\begin{equation*}
\begin{tikzpicture}[scale=.75, color=\clr, baseline=0]
	\node at (-6.5,1) {$\mu=$};
	\draw [thick ] (-5,1) to  (12,1);
	\draw (0,1)\bdot;
	\draw (1,1)\bdot;
	\draw (2,1)\bdot;
	\draw (3,1)\bdot;
	\draw (4,1)\tick;
	\draw (5,1)\tick;
	\draw (6,1)\tick;
	\draw (7,1)\tick;
	\draw (8,1)\bdot;
	\draw (9,1)\bdot;
	\draw (10,1)\bdot;
	\draw (11,1)\bdot;
        \node at (0,1.5) {$0$};
        \node at (8,1.5) {$8$};
	\node at (12.5,1) {$\cdots$};
	\node at (-5.5,1) {$\cdots$};
	\node at (-6.5,-1) {$\lambda=$};
	\draw [thick ] (-5,-1) to  (12,-1);
	\draw (-4,-1)\bdot;
	\draw (-3,-1)\bdot;
	\draw (-2,-1)\tick;
	\draw (-1,-1)\tick;
	\draw (0,-1)\bdot;
	\draw (1,-1)\bdot;
	\draw (2,-1)\tick;
	\draw (3,-1)\tick;
	\draw (4,-1)\bdot;
	\draw (5,-1)\bdot;
	\draw (6,-1)\tick;
	\draw (7,-1)\tick;
	\draw (8,-1)\bdot;
	\draw (9,-1)\bdot;
        \node at (-4,-1.5) {$-4$};
        \node at (0,-1.5) {$0$};
        \node at (4,-1.5) {$4$};
        \node at (8,-1.5) {$8$};
	\node at (12.5,-1) {$\cdots$};
	\node at (-5.5,-1) {$\cdots$};
      \draw [ thick, -> ]  (-0.2,0.85) to  (-3.8,-.85);
      \draw [ thick, -> ]  (0.8,0.85) to (-2.8,-.85);
      \draw [ thick, -> ]  (1.85,.8) to (0.15,-.8);
      \draw [ thick, -> ]  (2.85,0.85) to  (1.15,-.85);
      \draw [ thick, -> ]  (8,0.8) to  (8,-.8);
      \draw [ thick, -> ] (9,0.8) to (9,-.8);
      \draw [thick, -> ] (9.8,.8) to (4.2,-.8);
      \draw [thick, -> ] (10.8,.8) to (5.2,-.8);
\end{tikzpicture}
\end{equation*}
We have $\ell(\la,\mu)=24$, and both functions have $\mathtt{L}(f)=4$. So these functions correspond to two occurrences of $P(\la)$ in degree $d=8$.
\end{example}

\section{Complexity}\label{S:growth}

\subsection{Rates of Growth}\label{SS:ratesofgrowth}

Given a sequence of nonnegative integers $(r_{d})_{d \geq 0}$, the \emph{rate of growth} of the sequence is the smallest nonnegative integer $c$ for which there exists a fixed real number $K>0$ such that  $r_{d} \leq Kd^{c-1}$ for all $d \geq 0$.  If no such $c$ exists, then we declare the rate of growth of the sequence to be infinite.

We define the \emph{$z$-complexity} of any sequence $\{U_{d}\}_{d\ge 0}$ of finite-dimensional $\fg$-modules, $c_{z}(U_{\bullet})$, as the rate of growth of the number  of indecomposable summands of $U_{d}$.  We define the \emph{$z$-complexity} of a $\fg$-module $M\in\F$, $c_{z}(M)$, to be the $z$-complexity of the minimal projective resolution of $M$ in $\FF$.    Clearly $c_{z}(\Delta(\mu))\leq c_{z}(P_{\bullet}(\mu))$.

Similarly, define the \emph{complexity} of any sequence $\{U_{d}\}_{d\ge 0}$ of modules, $c(U_{\bullet})$, as the rate of growth of the dimensions (as vector spaces) of the $U_{d}$.  For $\mu\in X^{+}(T)$, write $c(\mu)$ for the complexity of the projective resolution $P_{\bullet}(\mu)$ of $\Delta(\mu)$ constructed in \cref{S:TheAlgorithm}. We define the \emph{complexity} of a $\fg$-module $M\in\F$,  $c_{\F}(M)$, to be the complexity of the minimal projective resolution of $M$ in $\FF$.  Clearly $c_{\F}(\Delta(\mu)) \le c(\mu)$. 

\subsection{z-Complexity}

  For $\mu\in X^{+}(T)$, let $s_{d}(\mu)$ denote the number of indecomposable summands in $P_{d}(\mu)$, where 
$P_{\bullet}(\mu) \to \Delta (\mu)$ is the projective resolution constructed in \cref{T:BBProjResolution}.   In this section we compute the $z$-complexity of this projective resolution; that is, the rate of growth of  the sequence $\left(s_{d}(\mu) \right)_{d \geq 0}$.

For $\mu\in X^{+}(T)$, define a \emph{run} of $\mu$ to be a maximal sequence of (one or more) adjacent dots in the weight diagram $[\mu]$. The \emph{size} of a run is the number of dots it contains (a run of size one is the same as an isolated dot). Let $\pi=\pi(\mu)=(\pi_{1},\dots,\pi_{t})$ be the sequence of sizes of the runs of $\mu$, ordered from \emph{right to left} in the weight diagram. Then $\pi$ is a composition of $n$. Let $o$ be the number of odd parts in $\pi$.

\begin{lemma} \label{L:separationindependent}
Suppose $\mu$ and $\tilde\mu$ are two dominant weights which differ only in the number of integers separating their runs. Then 
\[
s_{d}(\mu)=s_{d}(\tilde\mu)
\]
for all $d\ge 0$.
\end{lemma}

\begin{proof} Self evidently, $\atyp (\mu) = \atyp (\tilde{\mu})$. 

If $\mu$ and $\tilde\mu$ are typical, then the result is clear. It is also clear for $d=0$ for any weights $\mu$, $ \tilde{\mu}$. We proceed inductively by assuming the result is true up to degree $d-1$ for weights of the same atypicality as $\mu$ and $\tilde{\mu}$, and up to degree $d$ for weights of smaller atypicality. In particular, we assume that $\atyp(\mu)=\atyp(\tilde\mu) \ge 1$ and $d>0$. Let us say that the $k$th black dots in the weight diagrams of $\mu$ and of $\mu'$, scanning from the left, \emph{correspond}, for each $1\le k \le n$.

As a preliminary observation, suppose a weight diagram $\nu$ is obtained from $\tilde\nu$ by moving a single isolated dot from $k-2$ to $k-1$ where it remains isolated (i.e., there is no dot at $k-3$ or $k$ in either $\nu$ or $\tilde\nu$). Then $\Theta_{k} \Delta(\tilde\nu) = \Delta(\nu)$, $\Theta_{k} P_{d}(\tilde\nu) = P_{d}(\nu)$, and, since (by \cref{L:IsolatedDots}) any $P(\tilde\lambda)$ which appears in $P_{d}(\tilde\nu)$ will have a left-isolated dot at $k-2$, by \cref{T:TranslationonProjectives} $\Theta_{k}P(\tilde\lambda)$ will be an indecomposable projective summand of $P_{d}(\nu)$. Thus $s_{d}(\nu)=s_{d}(\tilde\nu)$, for all $d\ge 0$.

Now suppose $P_{d}(\mu)$ is obtained via Step 2a of the Algorithm. Then there are dominant weights $\nu,\mu'$ with $\atyp(\nu)<\atyp(\mu)$ and $\atyp(\mu')\le\atyp(\mu)$ and an index $j$ such that
$$
P_{d}(\mu) = \Theta_{j+1}P_{d}(\nu) \oplus P_{d-1}(\mu').
$$
Moreover, for the same reasoning as in the previous paragraph, $\Theta_{j+1}$ sends each indecomposable projective summand of $P_{d}(\nu)$ to an indecomposable projective (i.e., not zero). Thus 
\begin{equation}\label{E:sformula}
s_{d}(\mu)=s_{d}(\nu) + s_{d-1}(\mu').
\end{equation}

With one exception, $P_{d}(\tilde\mu)$ is also obtained via Step 2a of the Algorithm, by moving the dot at $\tilde j$, corresponding to the dot at $j$ in $\mu$, one position left to obtain the weight $\tilde\nu$, and subsequently moving the dot at $\tilde j+1$ one position left to obtain $\tilde\mu'$. As above, $s_{d}(\tilde\mu)=s_{d}(\tilde\nu) + s_{d-1}(\tilde\mu')$. Moreover $\tilde\nu$ (resp.\ $\tilde\mu'$) differs from $\nu$ (resp.\ $\mu'$) only in the number of integers separating their runs. Thus by our induction hypothesis, $s_{d}(\tilde\nu)=s_{d}(\nu)$ and $s_{d-1}(\tilde\mu')=s_{d-1}(\mu')$. Combined with the previous equations, this gives $s_{d}(\tilde\mu)=s_{d}(\mu)$. 

The exception occurs when $\tilde\mu$ has a dot at $\tilde j -2$. In this case a sequence of isolated dots at and possibly left of position $\tilde j-2$ in $\tilde\mu$ must be moved one position left, while remaining isolated, before Step 2a can be applied. However, according to the preliminary observation above, each weight in this sequence will have all the same values of $s_{d}$ as does $\tilde\mu$. Thus when we are finally able to apply Step 2a, we get the same conclusion as in the previous paragraph.

Suppose $P_{d}(\mu)$ (and/or $P_{d}(\tilde\mu)$) is obtained via Step 2b of the Algorithm. This step only involves moving isolated dots so that they remain isolated. But by the preliminary observation, this does not change $s_{d}$. Thus we may without loss assume that all necessary applications of Step 2b have been performed already, thereby reverting to the situation of Step 2a, where the result has been proved.
\end{proof}

\begin{lemma}\label{C:IndependentofWeight}  If $\mu$ and $\tilde{\mu}$ are dominant weights with run sizes $\pi = (\pi_{1}, \dotsc , \pi_{t})$ and $\tilde{\pi}=(\tilde{\pi}_{1}, \dotsc , \tilde{\pi}_{t})$ where $\pi$ and $\tilde{\pi}$ are equal as unordered multisets, then 
\[
s_{d}(\mu) = s_{d}(\tilde{\mu})
\] for all $d\geq 0$. 

\end{lemma}

\begin{proof}  Self evidently, $\atyp (\mu) = \atyp (\tilde{\mu})$.  We induct on atypicality  and $d$.  If $\mu$ and $\tilde{\mu}$ are typical, then the result is obvious for all $d \geq 0$; in general it is obvious for $d=0$.  Now consider the case when $\mu$ and $\tilde{\mu}$ are atypical and $d >0$. We assume as usual that the result holds in degree $d$ for all dominant weights which have strictly smaller atypicality, and in degree $d-1$ for all weights whose atypicality equals that of $\mu$ and $\tilde{\mu}$.  Since we are assuming $t < n$, it follows $\mu$ contains a run consisting of two or more dots, say the one indexed by $a$, of size $\pi_{a}$.  By assumption $\mu$ and $\tilde{\mu}$ have the same run sizes, so there is a $b$ such that $\tilde\pi_{b}=\pi_{a}$.  By \cref{L:separationindependent} we may also assume for both $\mu$ and $\tilde{\mu}$ that they have a large number of integer positions separating their runs.   By \cref{L:IndependenceOfOrder} we may also assume the construction of $P_{d}(\mu)$ (resp.\ $P_{d}(\tilde{\mu})$) was by applying Step 2a of the Algorithm to the leftmost dot of run $a$ (resp.\ $b$).   

As argued in the proof of \cref{L:separationindependent} to obtain \cref{E:sformula},  there are dominant weights $\nu$, $\mu'$ with $\atyp(\nu)<\atyp(\mu)$ and $\atyp(\mu')\le\atyp(\mu)$ such that
\[
s_{d}(\mu)=s_{d}(\nu) + s_{d-1}(\mu').
\] There are also dominant weights $\tilde{\nu}$, $\tilde{\mu}'$ with $\atyp(\tilde{\nu})<\atyp(\tilde{\mu})$ and $\atyp(\tilde{\mu}')\le\atyp(\tilde{\mu})$ such that
\[
s_{d}(\tilde{\mu})=s_{d}(\tilde{\nu}) + s_{d-1}(\tilde{\mu}').
\]  By the inductive assumption $s_{d}(\nu)=s_{d}(\tilde{\nu})$ and $s_{d-1}(\mu')= s_{d-1}(\tilde{\mu}')$, therefore $s_{d}(\mu)=s_{d}(\tilde{\mu})$ as desired.
\end{proof}

\subsection{Hilbert-Poincar\'e  Series for z-Complexity}\label{SS:Hilbertseries}
Given a dominant integral weight $\mu\in X^{+}(T)$ with run sizes $\pi = (\pi_{1}, \dotsc , \pi_{t})$, define a generating function 
\begin{equation}\label{E:generatingfunctiondef}
S_{\mu}(u) = S_{\pi}(u) = \sum_{d\geq 0} s_{d}(\mu)u^{d}.
\end{equation}
By the previous lemma this depends only on the composition $\pi$ and, indeed, only on the multiset of run sizes.  In what follows it will be convenient to write $S_{\pi}(u)$ for any tuple of positive integers $\pi =(\pi_{1}, \dotsc , \pi_{t})$, where the series is understood to be defined as in \cref{E:generatingfunctiondef} using a dominant weight $\mu$ with tuple of run sizes equal to $\pi$.  To avoid clutter we sometimes write $r$ for the composition $(r)$. 

For $r \geq 0$, let $f_{r}(u)$ be the polynomial in the variable $u$ determined by $f_{0}(u)=1$, $f_{1}(u)=1$, and the recursions:
\begin{align*}
f_{2k}(u)  &=(1-u)f_{2k-1}(u)+uf_{2k-2}(u), \\
f_{2k+1}(u) &= f_{2k}(u) + uf_{2k-1}(u),
\end{align*} for $k\ge 1$. In particular, $f_{2}(u)=1$. If $\pi = (\pi_{1}, \dotsc , \pi_{t})$ is a composition, then set 
\[
f_{\pi}(u) = f_{\pi_{1}}(u)\dotsb f_{\pi_{t}}(u).
\]  We write $o=o(\pi)$ for the number of odd runs (i.e.\  parts) in the composition $\pi$.

\begin{theorem}\label{T:zcomplexity}  The following statements hold true.
\begin{enumerate}
\item  If $\pi = (\pi_{1}, \dotsc , \pi_{t})$ is a composition of $n$, then
\[
S_{\pi}(u) = \prod_{j=1}^{t} S_{\pi_{j}}(u).
\]
\item For all $r \geq 0$,
\[
S_{r}(u) = \frac{f_{r}(u)}{\left(1-u \right)^{\lfloor r/2 \rfloor}}.
\]
\item  If $\pi = (\pi_{1}, \dotsc , \pi_{t})$ is a composition of $n$, then
\[
S_{\pi}(u) =  \frac{f_{\pi}(u)}{\left(1-u \right)^{\frac{n-o(\pi)}{2} }}.
\]
\item The $z$-complexity of the projective resolution $P_{\bullet}(\mu) \to \Delta (\mu)$ constructed in \cref{T:BBProjResolution} is given by 
\[
c_{z}\left(P_{\bullet}(\mu) \right)=\frac{n-o(\mu)}{2}.
\]

\end{enumerate}
\end{theorem}

\begin{proof}  We first prove (1) under the assumption that all runs for $\pi$ have size one or two.  In this case we prove the statement by inducting on the number of runs of size two.  The base case is when all runs have size one, in which case the $\mu$ for $\pi$ is typical and, hence, $S_{\pi}(u)=1$, $S_{1}(u)=1$, and the result follows trivially.  Now assume $\mu$ has at least one run of size two.  By \cref{L:IndependenceOfOrder} we may assume without loss that the leftmost run has size two; that is, $\pi_{t}=2$ (recall that our convention is to list the run sizes from right to left). By applying the Algorithm to the leftmost dot on this run we may apply \cref{E:sformula} and \cref{L:separationindependent} to deduce 
\[
S_{\pi}(u) = S_{(\pi_{1}, \dotsc , \pi_{t-1}, 1,1)}(u) +uS_{\pi}(u).
\]  Thus by the inductive assumption and the base case it follows that
\[
S_{\pi}(u) = \frac{1}{1-u}\prod_{i=1}^{t-1} S_{\pi_{i}}(u).
\]  But the same argument applied in the special case when $\pi = 2$ shows $S_{2}(u)=\frac{1}{1-u}$ and, hence, 
\[
S_{\pi}(u) = \prod_{i=1}^{t} S_{\pi_{i}}(u).
\]

We now prove (1) in general by inducting on the atypicality of $\mu$.  If $\mu$ is typical or consists of only runs of size one and two, it is handled by the previous paragraph.  Therefore we may assume $\pi$ has a run of size strictly greater than $2$.  Again by \cref{L:IndependenceOfOrder} we may assume this run is $\pi_{t}$.  As above, we have  
\[
S_{\pi}(u) = S_{(\pi_{1}, \dotsc , \pi_{t-1}, \pi_{t}-1,1)}(u) +uS_{(\pi_{1}, \dotsc , \pi_{t-1}, \pi_{t}-2, 2)}(u).
\] By the inductive assumption we have
\begin{equation*}
S_{\pi}(u) = \prod_{i=1}^{t-1} S_{\pi_{i}}(u)\left[S_{(\pi_{t}-1,1)}(u)+uS_{(\pi_{t}-2, 2)}(u) \right].
\end{equation*}
However, applying \cref{E:sformula} to the leftmost dot in the special case of a single run of size greater than two proves
\[
S_{\pi_{t}}(u)=S_{(\pi_{t}-1,1)}(u)+uS_{(\pi_{t}-2, 2)}(u).
\]  Substituting this into the previous equation proves (1).

To prove (2) we induct on $r$ with $r=1,2$ already handled earlier in the proof.  Therefore we may assume $r \geq 3$.  Arguing as in the previous paragraph we have the first equality below and the subsequent equalities follow by the induction assumption and the definition of $f_{r}(u)$:
\begin{align*}
S_{r}(u) & =S_{(r-1,1)}(u)+uS_{(r-2, 2)}(u)\\
      & = \frac{f_{r-1}(u)}{(1-u)^{\lfloor (r-1)/2 \rfloor}} + u\frac{f_{r-2}(u)}{(1-u)^{\lfloor (r-2)/2 \rfloor + 1}} \\
      & = \begin{cases} \frac{f_{r-1}(u)+uf_{r-2}(u)}{(1-u)^{\lfloor r/2 \rfloor}}, &\text{$r$ odd}; \\
                        \frac{(1-u)f_{r-1}(u)+uf_{r-2}(u)}{(1-u)^{\lfloor r/2 \rfloor}}, &\text{$r$ even};
\end{cases}\\
       & = \frac{f_{r}(u)}{(1-u)^{\lfloor r/2 \rfloor}}.
\end{align*}

Now (3) follows immediately from (1) and (2) and the observation that
\[
\sum_{i=1}^{t} \lfloor \pi_{i}/2 \rfloor = \frac{n-o(\pi)}{2},
\]  for any composition $\pi = (\pi_{1}, \dotsc , \pi_{t})$ of $n$.

Finally, the recursion formulas easily imply $f_{r}(1) > 0$ for all $r \geq 0$ and this along with (3) immediately implies (4) (e.g.\ see \cite[Lemma 4.1.7]{BrunsHerzog}). 
\end{proof}

\begin{remark} Using standard techniques one can determine the polynomials $f_{r}(u)$ defined recursively above.   Set $p(u)=-3u^{2}+2u+1$.  Then:

\begin{align*}
2^{k}f_{2k}(u)&=   \sum_{\substack{0\le i\le k\\ i \text{ even}}} \binom{k}{i} (u+1)^{k-i}\, p(u)^{i/2} +  \sum_{\substack{0\le i\le k\\ i \text{ odd}}} \binom{k}{i} (u+1)^{k-i}(1-u)\, p(u)^{(i-1)/2}, \\
2^{k}f_{2k+1}(u)&=  \sum_{\substack{0\le i\le k\\ i \text{ even}}} \binom{k}{i} (u+1)^{k-i}\, p(u)^{i/2} +  \sum_{\substack{0\le i\le k\\ i \text{ odd}}} \binom{k}{i} (u+1)^{k-i+1}\, p(u)^{(i-1)/2}.
\end{align*}
Alternatively, they can be written as single sums after reindexing:
\begin{align*}
2^{k}f_{2k}(u) &= \sum_{\substack{0 \leq i \leq k \\ i \;  \text{even}}} \left[\binom{k}{i}(u+1)+\binom{k}{i+1}(1-u) \right]\left(u+1 \right)^{k-i-1}p(u)^{i/2},   \\
2^{k}f_{2k+1}(u) &=  \sum_{\substack{0 \leq i \leq k \\ i \; \text{even}}} \binom{k+1}{i+1}(u+1)^{k-i}p(u)^{i/2}.
\end{align*}
As we do not need the precise form of the polynomials for the paper, we omit the derivation.
\end{remark}

\subsection{Complexity}\label{SS:complexity}  We next consider the complexity $c(\mu)$ of the resolution
$P_{\bullet}(\mu) \to \Delta (\mu)$  constructed in \cref{T:BBProjResolution} for any $\mu \in X^{+}(T)$.  

Using the matrix realization given in \cref{E:matrixform} we identify $\fg_{-1}$ as the space of skew-symmetric $n\times n$ matrices.   Let $(\fg_{-1})_{k}$ be the subset of $\fg_{-1}$ consisting all matrices of rank $k$.  Its Zariski closure $\overline{(\fg_{-1})_{k}}$ then consists of the elements of $\fg_{-1}$ of rank at most $k$. It is an easy exercise to verify when $k=2\ell$ is even, $\dim \overline{(\fg_{-1})_{2\ell}} = \ell(2n-2\ell-1)$. 

\begin{theorem}\label{T:complexityupperbound}
For $\mu\in X^{+}(T)$ let $o$ denote the number of odd runs of $\mu$.   We then have
$$
c(\mu) \le  \binom{n}{2} - \binom{o}{2} = \dim \overline{(\fg_{-1})_{n-o}}.
$$
\end{theorem}

\begin{proof}
The equality is a simple consequence of the dimension formula immediately preceding the theorem, so we focus on the inequality. We need an upper bound on the dimensions of the possible indecomposable projective summands $P(\la)$ of $P_{d}=P_{d}(\mu)$. The argument in \cite[Section 5.1]{BKN4} carries over {\it mutatis mutandis} to show that
$$
\dim L_{0}(\la) \le \dim P(\la) \le 2^{\dim \fg_{\1}} \dim L_{0}(\la).
$$
As in \cite[Section 5.2]{BKN4}, it suffices to obtain an upper bound (as a monomial in $d$) for $\dim L_{0}(\la)$ for the possible direct summands $P(\la)$ of $P_{d}$. Fix such a summand and let $f:\mu\to\la$ be the associated allowable function as in \cref{T:degree}. Since $f$ is a bijection, we have a permutation $\varphi\in S_{n}$ such that $f(\bar\mu_{i}) = \bar\la_{\varphi(i)}$ for all $i$; recall that $\bar{\mu}=\mu+\rho$. 

First, by tensoring by sufficiently many copies of the supertrace representation we may assume that $\bar{\mu}_{i}\le 0$ for $1\le i\le n$. Also $\la\ge\mu$, whence $\bar\la_{i}\le\bar\mu_{i}\le 0$ for all $i$.

Next, it is an immediate consequence of the Moves in \cref{SS:KeyFunctions} that each dot of $\mu$ is either fixed by $f$, or is part of an ``adjacent pair'' of dots to which Move 3 was applied at some step of the algorithm. In fact, it is not hard to see that the first time Move 3 is applied to a particular pair of adjacent dots, they must both have been fixed, and after the move, they each map one  position left via the allowable function. All subsequent moves keep this pair of dots adjacent (although their images under the function may become separated by applications of Move 2). We view such a pair as permanently ``linked'' through the remainder of the algorithm, and call them an ``adjacent pair.''

In particular, each of the $o$ odd runs of $\mu$ has at least one fixed point $\bar\mu_{i}=f(\bar\mu_{i})=\bar\la_{\varphi(i)}$. Set
\begin{gather*}
F_{0}(\mu,f) = \{\,1\le i\le n  \mid  \bar\mu_{i} \text{ is the rightmost fixed point of an odd run of } \mu\,\},\\
F'_{0}(\la,f) = \{\,\varphi(i)  \mid  i\in F_{0}(\mu,f)\,\},
\end{gather*}
Note that these are each sets of size $o$.
Evidently the remaining dots of each (even or odd) run of $\mu$ (other than those indexed by $F_{0}(\mu,f)$) consist of adjacent pairs and (an even number of) fixed points. Let's pair these (remaining) fixed points beginning at the left of each run. Set
\begin{gather*}
P(\mu,f)=\{\,(i,i+1)  \mid  \bar\mu_{i} \text{ and } \bar\mu_{i+1} \text{ are an adjacent pair in some run of } \mu\,\},\\
F(\mu,f)=\{\,(i,j)  \mid  i<j,\ \bar\mu_{i} \text{ and } \bar\mu_{j} \text{ are paired fixed points in some run of } \mu\,\}.
\end{gather*}
Define $P'(\la,f)$ (resp.\ $F'(\la,f)$) by applying $\varphi$ to each pair in $P(\mu,f)$ (resp.\ $F(\mu,f)$). Then each index $1\le i\le n$ appears exactly once in one of $F_{0}(\mu,f),\ P(\mu,f)$, or $F(\mu,f)$ (resp.\ $F'_{0}(\la,f), P'(\la,f)$, or $F'(\la,f)$).

Partition the positive even roots $\Phi^{+}_{0} = \{\,\ve_{r}-\ve_{s}  \mid  r < s \,\}$ into three subsets:
\begin{gather*}
A(\mu) = \{\,\ve_{r}-\ve_{s}  \mid  r, s \in F_{0}(\mu,f) \,\},\\
B(\mu) = \{\,\ve_{r}-\ve_{s}  \mid  (r,s)\in P(\mu,f) \cup F(\mu,f) \,\}, \\
C(\mu) = \Phi^{+}_{0} \smallsetminus (A(\mu) \cup B(\mu)).
\end{gather*}
Define $A'(\la), B'(\la), C'(\la)$ similarly using $F'_{0}(\la,f), P'(\la,f), F'(\la,f)$. Notice that $\# A(\mu)=\# A'(\la)\linebreak[0]=\binom{o}{2}$ and $\# B(\mu)=\# B'(\la)= (n-o)/2$.

To bound the dimension of $L_{0}(\lambda)$ we use the Weyl dimension formula for $\gln$:
\begin{equation}\label{E:WCF}
\dim L_{0}(\la) = \prod_{\a\in\Phi^{+}_{0}}\frac{(\la+\rho,\a)}{(\rho,\a)} = \prod_{\a\in A'(\la)}\frac{(\la+\rho,\a)}{(\rho,\a)} \prod_{\a\in B'(\la)}\frac{(\la+\rho,\a)}{(\rho,\a)} \prod_{\a\in C'(\la)}\frac{(\la+\rho,\a)}{(\rho,\a)}.
\end{equation}
Since the denominators in \cref{E:WCF} are at least one, we can and will henceforth ignore them in determining an upper bound.

We now analyze each of the three factors on the right hand side of \cref{E:WCF}. The numerator in the first factor is a product over pairs $r'<s'$ in 
$F'_{0}(\la,f)$ of
$$
(\bar\la,\ve_{r'}-\ve_{s'})=\bar\la_{r'}-\bar\la_{s'}=\bar\mu_{r} - \bar\mu_{s},
$$
where for simplicity of notation we are denoting $\varphi^{-1}(r')$ by $r$, and similarly for $s$.
Thus the first factor in \cref{E:WCF} is bounded above by a positive constant $C_{1}$ depending only on $n$ and $\mu$.

The numerator of the second factor is a product of two types of factors $\bar\la_{r'}-\bar\la_{s'}$. The first type, coming from pairs of fixed points $(r,s)\in F(\mu,f)$, can be handled exactly as in the previous paragraph. The second type comes from pairs $(r,s)\in P(\mu,f)$. In this case $r$ and $s=r+1$ index an adjacent pair in $\mu$. Recall that these are created via an application of Move 3 in the algorithm, with the image dots also initially adjacent. Subsequently the image dots under the function can be moved apart by 2 for each instance of Move 2: an isolated dot moving from left to right past the leftmost of the two image dots. Since there are at most $n-2$ dots to the left which can effect such a move, it follows that 
$$
\bar\la_{r'}-\bar\la_{s'} \le 2(n-2)+(\bar\mu_{r}-\bar\mu_{s}).
$$
In particular, the second factor in \cref{E:WCF} is bounded above by a positive constant $C_{2}$ depending only on $n$ and $\mu$.

Next, we consider the third factor. Recall the identity 
$$
d= \tfrac12 \ell(\lambda, \mu) - \mathtt{L}(f).
$$
from \cref{T:degree}. Here $\mathtt{L}(f) = \#\{\, i<j  \mid  f(\bar\mu_{i}) > f(\bar\mu_{j}) \,\} \le \binom{n}{2}$. Thus
$$
\sum_{i}(\bar\mu_{i}-\bar\la_{i}) =  \ell(\la,\mu) \le 2d + 2\binom{n}{2}.
$$
Since we normalized so that, for each $1\le i\le n$, $\bar\la_{i}\le\bar\mu_{i}\le 0$, we deduce
$$
|\bar\la_{i}| = -\bar\la_{i} \le -\bar\mu_{i} + 2d + n(n-1) = 2d + |\bar\mu_{i}| + n(n-1).
$$
Hence a typical factor in the numerator of the third factor of \cref{E:WCF} is
$$
\bar\la_{r'}-\bar\la_{s'}\le |\bar\la_{r'}| + |\bar\la_{s'}| \le 4d + D
$$
where $D$ is a positive constant depending only on $n$ and $\mu$. Since $\# C'(\la)=\binom{n}{2} - \binom{o}{2} - \frac{n-o}{2}$, we deduce that the third factor of \cref{E:WCF} is bounded above by 
$$
C_{3} d^{\binom{n}{2} - \binom{o}{2} - \frac{n-o}{2}},
$$
where $C_{3}$ is a positive constant depending only on $n$ and $\mu$.

Putting this all together, we have that $\dim P(\la)\le C_{4} d^{\binom{n}{2} - \binom{o}{2} - \frac{n-o}{2}}$ where $C_{4}=C_{1}C_{2}C_{3}\cdot 2^{\dim \fg_{\1}} $ is a positive constant depending only on $n$ and $\mu$. Combining this with the computation of the $z$-complexity of this resolution in \cref{T:zcomplexity} we have
$$
\dim P_{d} = \sum_{P(\la) \mid P_{d}} \dim P(\la) \le C d^{\binom{n}{2} - \binom{o}{2} - \frac{n-o}{2}} d^{\frac{n-o}{2} - 1} = C d^{\binom{n}{2} - \binom{o}{2} - 1},
$$
where $C$ is a positive constant depending only on $n$ and $\mu$. This gives the inequality in the statement of the theorem, and completes the proof.
\end{proof}

\section{Support and Associated Varieties} \label{S:supportvarieties}

\subsection{Support Varieties}\label{SS:introtosupports}

Let $\mathfrak{g}$ be a classical Lie superalgebra and let $M$
be in ${\mathcal F}:={\mathcal F}(\g,\g_{\0})$. According to \cite{BKN1}, 
$R:=\operatorname{H}^{\bullet}(\mathfrak{g}, \mathfrak{g}_{\0};{\mathbb C})= \oplus_{d \geq 0} \Ext_{\FF}^{d}(\C , \C)$ is a finitely generated commutative ring and 
$\Ext_{\mathcal{F}}^{\bullet}(M,M) = \oplus_{d \geq 0} \Ext_{\FF}^{d}(M , M)$ is a finitely generated $R$-module.
Set $J_{({\mathfrak g},{\mathfrak g}_{\0})}(M) :=
\operatorname{Ann}_{R}(\Ext_{\mathcal{F}}^{\bullet}(M,M))$
(i.e., the annihilator ideal of this module).  The \emph{support variety of $M$} is
\begin{equation}
\mathcal{V}_{(\mathfrak{g},\mathfrak{g}_{\0})}(M) :=
\operatorname{MaxSpec}(R/J_{({\mathfrak g},{\mathfrak g}_{\0})}(M)).
\end{equation}

Since $\fg = \fp (n)$ is $\Z$-graded and concentrated in degrees $-1$, $0$, and $1$, both $\fg_{1}$ and $\fg_{-1}$ are abelian Lie superalgebras.  Consequently, 
\[
R_{\pm} :=\HH^{\bullet} (\fg_{\pm 1}, \C ) = \HH^{\bullet}(\fg_{\pm 1}, \{0 \};\C) \cong S(\fg_{\pm 1}^{*})
\] as graded algebras.  Let $\mathcal{F}(\fg_{\pm 1})$ be the category of finite-dimensional $\fg_{\pm 1}$-modules.  If $M$ is an object in $\mathcal{F}(\fg_{\pm 1})$, then one can define the $\fg_{\pm 1}$ \emph{support variety} of $M$, 
\[
\V_{\fg_{\pm 1}}(M) := \V_{(\fg_{\pm 1},0)}(M),
\] as above.  Since $\fg_{\pm 1}$ is abelian the arguments given in \cite[Section 5]{BKN1} for detecting subalgebras apply here as well and one has that $\V_{\fg_{\pm 1}}(M)$ is canonically isomorphic to the following rank variety: 
\[
\V_{\fg_{\pm 1}}^{\text{rank}}(M) :=\left\{ x \in \fg_{\pm 1}  \mid  M \text{ is not projective as a $U(\langle x \rangle)$-module} \right\} \cup \{ 0 \},
\] where $U(\langle x \rangle)$ denotes the enveloping algebra of the Lie subsuperalgebra generated by $x \in \fg_{\pm 1}$.  We will identify $\V_{\fg_{\pm 1}}(M)$ and $\V^{\text{rank}}_{\fg_{\pm 1}}(M)$ via this canonical isomorphism.

Let $\fp^{\pm} = \fg_{0} \oplus \fg_{\pm 1}$ and let $M$ be a $\fp^{\pm}$-module in $\F_{(\fp^{\pm},\fg_{0})}$.  By \cite[Theorem 3.3.1]{BKN4}, we have

\begin{equation}\label{E:complexityequalssupportsforp}
c_{\F(\fp^{\pm},\fg_{0})}(M)=\dim {\mathcal V}_{{\mathfrak g}_{\pm 1}}(M)=\dim \V_{\fg_{\pm 1}}^{\rm{rank}}(M).
\end{equation}

\subsection{A Rank Variety Calculation}\label{SS:rankvarietycalc} Before continuing we establish the existence of certain explicit elements in $\VV _{\fg_{-1}}(\Delta(\mu))$ for $\mu\in X^{+}(T)$.  Fix $k$ with $1 \leq k \leq  n$.  Let $\fg'=\fg'(k)$ be the subspace of $\fg = \fp (n)$ spanned by the following sets:
\begin{gather*}
\left\{h_{i}  \mid 1 \leq i \leq  n, i \neq k \right\}, \\
\left\{a_{i,j}  \mid  1 \leq i \neq j \leq  n, i \neq k, j \neq k \right\}, \\
\left\{b_{i,j}   \mid  1 \leq i \leq j \leq  n, i \neq k, j \neq k \right\}, \\
\left\{c_{i,j}   \mid 1 \leq i < j \leq  n, i \neq k, j \neq k \right\}.
\end{gather*}
That is, $\fg'$ is the subspace of $\fg$ consisting of matrices which have zeros in the $k$th and $(n+k)$th rows and columns.  Using the commutator formulas given in \cref{SS:conventions} we see $\fg'$ is a Lie subsuperalgebra of $\fg$ and, moreover, $\fg' \cong \fp (n-1)$.   

For this section, we write $X_{n}^{+}(T)$ for the dominant integral weights for $\fp (n)$. Similarly, for $\mu \in X_{n}^{+}(T)$ we write $L_{0,n}(\mu)$ for the simple $\fp (n)_{0} \cong \gl (n)$ module of highest weight $\mu$ and 
\begin{equation*}
\Delta_{n}(\mu) = U(\fp (n)) \otimes_{U(\fp(n)_{0}\oplus \fp(n)_{-1})} L_{0, n}(\mu )
\end{equation*} for the Kac module. 
Given $\mu = \sum_{i=1}^{n} \mu_{i}\varepsilon_{i} \in X_{n}^{+}(T)$, let $\mu ' = \sum_{i \neq k} \mu_{i}\varepsilon_{i}$ be identified in the obvious way with an element of $X_{n-1}^{+}(T)$.

\begin{lemma}\label{L:RestrictingKacModules} Fix $k$ with $1 \leq k \leq n$ and let $\fg = \fp (n)$ and $\fg' =\fg'(k) \subseteq \fg$.  For any $\mu \in X_{n}^{+}(T)$ let $\mu ' \in X_{n-1}^{+}(T)$ be as above. Then there is a direct sum decomposition of $\fg '$-modules,
\[
\Delta_{n}(\mu) \cong \Delta_{n-1}(\mu') \oplus U,
\] for some $\fg'$-module $U$.
\end{lemma}

\begin{proof}  First, consider $L_{0,n}(\mu)$ as a $\fg'_{0}$-module.  By complete reducibility and the fact that the highest weight vector in $L_{0,n}(\mu)$ is a highest weight vector of weight $\mu'$ for $\fg'$, it follows that 
\[
L_{0,n}(\mu) \cong L_{0,n-1}(\mu ') \oplus U'
\] for some $\fg'_{0}$-submodule $U'$.

The superspace $\Lambda^{\bullet}\left(\fg_{1} \right)$ has a basis consisting of monomials in the elements $b_{i,j}$ for $1 \leq i \leq j \leq n$. Let $\Lambda'$ denote the subspace spanned by the monomials in the elements $\left\{b_{i,j}  \mid  1 \leq i \leq  j \leq n, i, j \neq k \right\}$ and let $\Lambda''$ be the span of all monomials which contain at least one $b_{i,k}$ or $b_{k,i}$ for $1 \leq i \leq n$.  As a superspace, we have the following decomposition:
\[
\Lambda^{\bullet}\left(\fg_{1} \right) = \Lambda' \oplus \Lambda''.
\] From the PBW theorem we have 
\begin{equation}\label{E:gradedsum}
\Delta_{n}(\mu) = \Lambda^{\bullet}(\fg_{1}) \otimes L_{0, n}(\mu )
\end{equation}
as a superspace. Combining this with the decompositions given above yields the following decomposition as superspaces: 
\[
\Delta_{n}(\mu) = \Big(\Lambda' \otimes L_{0, n-1}(\mu')  \Big) \oplus \Big(  \Lambda' \otimes U' \oplus \Lambda'' \otimes L_{0,n}(\mu) \Big).
\]  Using the commutator formulas given in \cref{SS:conventions} it follows this is a decomposition as $\fg'$-modules.

It remains to verify $\Lambda' \otimes L_{0, n-1}(\mu')$ is isomorphic to $\Delta_{n-1}(\mu')$ as a $\fg'$-module.  Since $1 \otimes L_{0, n-1}(\mu')$ is isomorphic to  $ L_{0,n-1}(\mu') $ as a $\fg'_{0}\oplus \fg'_{-1}$-module, the universal property of $\Delta_{n-1}(\mu')$ implies there is a surjective $\fg '$-module homomorphism $\Delta_{n-1}(\mu') \to \Lambda' \otimes L_{0, n-1}(\mu')$.  As the dimensions of these superspaces coincide, it is an isomorphism.
\end{proof}

The module $\Delta(\mu)$ admits a $\Z$-grading determined by \cref{E:gradedsum}.  Namely, the $\Z$-grading is given by
 \[
\Delta (\mu) = \Delta (\mu)_{0} \oplus \Delta (\mu)_{1} \oplus \Delta (\mu)_{2} \oplus \dotsb, 
\] where 
\[
\Delta (\mu)_{d} = \Lambda^{d}(\fg_{1}) \otimes L_{0}(\mu).
\]  Moreover this $\Z$-grading is compatible with the $\Z$-grading of $\fg$ introduced in \cref{SS:LSAoftypeP} in the sense that $\fg_{r}.\Delta(\mu)_{s} \subseteq \Delta(\mu)_{r+s}$ for all $r,s \in \Z$.  

Given a homogeneous $x \in \fg$, we write $\langle x \rangle$ for the Lie subsuperalgebra generated by $x$.  In particular, if $x \in \fg_{-1}$, then $\langle x \rangle$ is a one-dimensional Lie superalgebra concentrated in odd parity and the enveloping superalgebra $U(\langle x \rangle)$ is isomorphic to an exterior algebra on one generator.  In this case it is well known that the only indecomposable modules (up to parity shift) are the trivial module and its projective cover, $U(\langle x \rangle )$.

In what follows, for a partition $\la$ write $\ell (\lambda)$ for the number of nonzero parts of $\lambda$, $2\la$ for the partition obtained by doubling all entries of $\la$, and $\la'$ for the conjugate (or transpose) of $\la$.  We write $c_{\la_{1}, \la_{2}}^{\la_{3}}$ for the Littlewood-Richardson coefficient defined by 
\[
c_{\la_{1}, \la_{2}}^{\la_{3}} = \left[ L_{0}(\la_{1}) \otimes L_{0}(\la_{2}): L_{0}(\la_{3}) \right].
\]

\begin{theorem}\label{T:rankvariety}  Let $\mu\in X^{+}(T)$ be fixed and set $o$  to be the number of odd runs in $\mu$.  Then there is a rank $n-o$ matrix $x \in \fg_{-1}$ such that $\Delta(\mu)$ is not free as a $U\left( \langle x \rangle\right)$-module.  In particular, $x \in \VV_{\fg_{-1}}(\Delta(\mu))$.
\end{theorem}

\begin{proof}  The proof is by induction on the number of odd runs in $\mu$. If $\mu$ has an odd run, let $\mu_{k}$ be an entry in such a run. Let $\fg'=\fg'(k)$ and $\mu'$ be as in \cref{L:RestrictingKacModules}. By the induction hypothesis there is an element $x\in\fg'_{-1}$ of rank $(n-1)-(o-1)$ such that $\Delta_{n-1}(\mu') \vert_{U(\langle x \rangle)}$ is not free. By \cref{L:RestrictingKacModules}, $x\in\fg_{-1}$ is an element of rank $n-o$ such that $\Delta_{n}(\mu) \vert_{U(\langle x \rangle)}$ is not free.

Thus it remains to check the base case $o=0$.
 With this assumption,  $n$ is even and we choose to reindex and write it as $2n$ for convenience.  Furthermore, by tensoring with the one-dimensional supertrace representation if needed, we may assume without loss that the $n$th coordinate of $\mu$ is equal to zero.  

Set 
\[
x = \left(\begin{matrix} 0 & 0 & 0 & 0 \\
                             0 & 0 & 0 & 0 \\
                             0 & I_{n} & 0 & 0 \\
                             -I_{n} & 0 & 0 & 0 
\end{matrix} \right) \in \fg_{-1},
\] where $I_{n}$ denotes the $n \times n$ identity matrix and $0$ denotes the $n \times n$ zero matrix.  Set 
\[
\mathfrak{c}_{x} = \left\{ a \in \fg_{\0 } \mid [a, x]=0 \right\}.
\] This is the subalgebra of $\fg_{\0}$ consisting of elements which centralize $x$.  By comparing with \cite[Section~1.2]{humphreysLieAlg} we see $\mathfrak{c}_{x} \cong \mathfrak{sp}(2n)$.  Explicitly, 
\[
\mathfrak{c}_{x} = \left\{ \left(\begin{matrix} m & q & 0 & 0 \\
                             p & -m^{t} & 0 & 0 \\
                             0 & 0 & -m^{t} & -p \\
                             0 & 0 & -q & m 
\end{matrix} \right) \right\},
\]
where $m, p, q$ are $n \times n$ matrices with $p$ and $q$ symmetric.

\medskip
\noindent  \textbf{Claim 1:} $\Delta(\mu)_{0} \cong L_{0}(\mu)$ contains a trivial direct summand as a $\mathfrak{c}_{x}$-module.

\smallskip
Since  $\mu \in X^{+}(T)$ is assumed to have $n$th coordinate equal to zero, there are unique partitions $\mu^{+}$ and $\mu^{-}$ with at most $n$ parts such that
\[
\mu = \sum_{i=1}^{\ell \left(\mu^{+} \right)} \mu^{+}_{i}\varepsilon_{i} - \sum_{i = 1}^{\ell \left(\mu^{-} \right)} \mu^{-}_{\ell \left(\mu^{-} \right)+1-i} \varepsilon_{2n+1-i}.
\]   Furthermore, since by assumption $\mu$ consisted of only even runs, the column lengths of $\mu^{+}$ and $\mu^{-}$ are necessarily even.
By \cite[Theorem~2.4.2]{HTW:05}, if $\nu$ is a partition with at most $n$ parts, then  the multiplicity of the simple $\mathfrak{c}_{x} \cong \mathfrak{sp}(2n)$-module of highest weight $\nu$ in the $\fg_{\0} \cong \mathfrak{gl}(2n)$-module $L_{0}(\mu)$ is given by 
\begin{equation}\label{E:LRcoefficients}
\sum_{\alpha, \beta, \gamma, \delta} c_{\alpha, \beta}^{\nu}\, c_{\alpha, (2\gamma)'}^{\mu^{+}}\, c_{\beta, (2\delta)'}^{\mu^{-}} \, ,
\end{equation} where the sum is over all partitions $\alpha$, $\beta$, $\gamma$, $\delta$.

Set $\nu = \alpha = \beta = \emptyset$.   For these choices for $\alpha$, $\beta$, $\gamma, \delta$, and $\nu$ it is easy to verify the Littlewood-Richardson coefficients appearing as factors of the corresponding term in \cref{E:LRcoefficients} are positive.  Consequently the trivial $\mathfrak{c}_{x}$-module appears as a direct summand of $L_{0}(\mu)$, as claimed.

\medskip
\noindent \textbf{Claim 2:} $\Delta(\mu)_{1} \cong \fg_{1}\otimes L_{0}(\mu)$ does not contain a trivial direct summand as a $\mathfrak{c}_{x}$-module.  

\smallskip
We first observe that as a $\fg_{\0}$-module, $\fg_{1} \cong S^{2}(V)$, where $V$ is the natural $\fg_{\0}\cong \gl (2n)$-module and $\Delta(\mu)_{1} \cong S^{2}(V) \otimes L_{0}(\mu)$ as a $\fg_{\0}$-module.  We then have 

\begin{multline}
\Hom_{\mathfrak{c}_{x}}\left( \Delta(\mu)_{1}, \C \right) \cong \Hom_{\mathfrak{c}_{x}}\left( S^{2}(V)\otimes L_{0}(\mu) , \C \right) \\ \cong \Hom_{\mathfrak{c}_{x}}\left( L_{0}(\mu), S^{2}(V)^{*} \right) \cong \Hom_{\mathfrak{c}_{x}}\left( L_{0}(\mu), S^{2}(V) \right),
\end{multline}
where the last isomorphism follows from the fact that $S^{2}(V) \cong S^{2}(V)^{*}$  as $\mathfrak{c}_{x}$-modules.  Therefore, $\Delta(\mu)_{1}$ contains a trivial direct summand as a $\mathfrak{c}_{x}$-module if and only if $S^{2}(V)$ is a direct summand of $L_{0}(\mu)$.  

Since $S^{2}(V)$ is a simple $\mathfrak{c}_{x}$-module we can again use \cref{E:LRcoefficients}, but now with $\nu = (2)$.  For $c_{\alpha, \beta}^{(2)}$ to be nonzero it must be that we are in one of three possible cases: (i) $\alpha = \emptyset$, $\beta = (2)$; (ii) $\beta = \emptyset$, $\alpha = (2)$; or (iii) $\alpha = \beta = (1)$.  We handle each case in turn.  For (i) we have $\alpha=\emptyset$ and $\beta = (2)$.  However, since $\mu^{-}$ and $(2\delta)'$ will both be partitions whose columns all have even length, the Littlewood-Richardson rule easily shows $c_{(2), (2\delta)'}^{\mu^{-}} =0$ regardless of $\delta$.  For (ii) one argues similarly and shows $c_{(2), (2\gamma)'}^{\mu^{+}} =0$ for all $\gamma$.  Finally, for (iii) an even easier application of the Littlewood-Richardson rule verifies $c_{(1), (2\gamma)'}^{\mu^{+}}=c_{(2), (2\delta)'}^{\mu^{-}}=0$ regardless of the choices of $\gamma$ or $\delta$.  Therefore, in every case $S^{2}(V)$ does not appear as a summand of $L_{0}(\mu)$ as a $\mathfrak{c}_{x}$-module and, hence, $\Delta (\mu)_{1}$ does not contain a trivial direct summand for $\mathfrak{c}_{x}$.

\medskip
We can now prove the statement of the theorem.  Since $U\left( \langle x \rangle  \right)$ is isomorphic to an exterior algebra on one generator, $\Delta(\mu)$ is isomorphic to a direct sum of modules isomorphic to $\C$ and $U(\langle x \rangle)$.  Consequently it suffices to exhibit a nonzero vector $v \in \Delta (\mu)$ such that $xv=0$ and to prove there does not exist a vector $w \in \Delta(\mu)$ for which $xw = v$.  Such a vector $v$ necessarily spans a trivial direct summand which, in turn, implies $x \in \VV_{\fg_{-1}}\left(\Delta(\mu) \right)$, as desired.

By Claim 1 we may choose a $v \in \Delta(\mu)_{0}$ which spans a trivial direct summand for $\mathfrak{c}_{x}$ and the $\Z$-grading on $\Delta (\mu)$ implies $xv=0$.  If there existed a $w \in \Delta(\mu)$ for which $xw =v$, then without loss we could assume $w \in \Delta(\mu)_{1}$ (again thanks to $\Z$-grading considerations).  Furthermore, since $\fc_{x}$ is semisimple and the action of $x$ defines a $\mathfrak{c}_{x}$-module homomorphism, Schur's Lemma would then imply $w$ must span a trivial $\fc_{x}$-module in $\Delta(\mu)_{1}$.  However, by Claim 2 no such vector can exist. 
\end{proof}

\subsection{Complexity and Support Varieties}\label{SS:complexityandsupports}

Recall we write $(\fg_{-1})_{k}$ for the rank $k$ matrices of $\fg_{-1}$; its closure $\overline{(\fg_{-1})_{k}}$ consists of the matrices of rank at most $k$. Let $G_{0} \cong \operatorname{GL}(n)$ with the adjoint action on $\fg_{\pm 1}$ and $\fg_{\1}$.

\begin{theorem}\label{T:complexityandsupportequalities}
Let $\mu$ be a dominant weight for $\fg =\fp (n)$ and let $o$ be the number of odd runs in $\mu$. Then 
\[
\V_{\fg_{-1}}(\Delta(\mu)) = \overline{(\fg_{-1})_{n-o}},
\]
and 
\[
c_{\F}(\Delta(\mu)) = \dim \overline{(\fg_{-1})_{n-o}} = \binom{n}{2} - \binom{o}{2}.
\]
Moreover the projective resolution of $\Delta(\mu)$ constructed in \cref{S:TheAlgorithm} has the same rate of growth as the minimal projective resolution.
\end{theorem}

\begin{proof}
The formula for $\dim \overline{(\fg_{-1})_{n-o}}$ was stated in \cref{T:complexityupperbound}. From \cref{T:rankvariety} and the rank variety description given in \cref{SS:introtosupports} it follows there exists  $x \in\V_{\fg_{-1}}(\Delta(\mu))$ of rank $n-o$.  Moreover, since the support variety of a $\fg$-module restricted to $\fg_{-1}$ is a closed, $G_{0}$-invariant subvariety of $\fg_{-1}$, it follows that $\overline{(\fg_{-1})_{n-o}} \subseteq \V_{\fg_{-1}}(\Delta(\mu))$. Using \cite[Section 6.1]{BKN4} together with \cref{E:complexityequalssupportsforp,T:complexityupperbound}, we deduce
$$
\dim \overline{(\fg_{-1})_{n-o}} \le \dim \V_{\fg_{-1}}(\Delta(\mu)) = c_{\F(\fp^{-},\fg_{0})}(\Delta(\mu)) \le c_{\F(\fg,\fg_{0})}(\Delta(\mu)) \le c(\mu)   \leq \dim \overline{(\fg_{-1})_{n-o}}.
$$
Hence all the inequalities are equalities, and the theorem follows.
\end{proof}

Let 
\[
\XX_{\C} = \left\{x \in \fg_{\1} \mid  [x,x]=0 \right\}
\] be the cone of odd self-commuting elements in $\fg$. For $M \in \F$,  Duflo and Serganova \cite{dufloserganova} introduced the \emph{associated variety} for $M$:
\[
\XX_{M} = \left\{x \in \XX_{\C} \mid M \text{ is not projective as a $U\!\left(\langle x \rangle \right)$-module} \right\} \cup \left\{0 \right\}.
\]  

\begin{theorem}\label{T:complexityandsupports}  Let $\mu \in X^{+}(T)$ be a dominant weight for $\fg =\fp (n)$ and let $o$ be the number of odd runs in $[\mu]$.  Then
\begin{enumerate}
\item  $\XX_{\Delta(\mu)} =  \overline{(\fg_{-1})_{n-o}} $;
\item $\VV_{(\fg , \fg_{\0})}\left(\Delta(\mu) \right) = \left\{0 \right\}$;
\item  $c_{\F}\left(\Delta(\mu) \right) = \dim  \XX_{\Delta(\mu)} + \dim \VV_{(\fg , \fg_{\0})}\left(\Delta(\mu) \right) $.
\end{enumerate}
\end{theorem}

\begin{proof}   A calculation entirely analogous to the proof of \cite[Theorem 6.4.1(a)]{BKN4} using the first equality in \cref{T:complexityandsupportequalities} proves (1). The argument used in the proof of \cite[Theorem 3.3.1]{BKN2} applies  and shows $\VV_{(\fg , \fg_{\0})}\left(\Delta(\mu) \right) = \left\{0 \right\}$ as claimed in (2).  Combining these results with the complexity calculation in \cref{T:complexityandsupportequalities} proves (3).
\end{proof}

\subsection{z-Complexity and Support Varieties for the Detecting Subalgebra}\label{SS:zcomplexityandfsupports}

For $\fg =\fp (n)$ set $h = \lfloor n/2 \rfloor$ and let $\ff_{\1}$ be the subspace of $\fg_{\1}$ spanned by $\left\{b_{i,h+i}  \mid 1 \leq i \leq h \right\} \cup \left\{ c_{i,h+i}  \mid  1 \leq i \leq h \right\}$.  Set $\ff_{\0}= \left[\ff_{\1}, \ff_{\1} \right] \subseteq \fg_{\0}$ and $\ff = \ff_{\0} \oplus \ff_{\1}$.  Then $\ff$ is a Lie subsuperalgebra of $\fg$ which in \cite{BKN1} was called a \emph{detecting subalgebra}\footnote{The careful reader will note that $\ff$ is actually a $G_{\0}$-conjugate of the detecting subalgebra introduced in \cite{BKN1}.  This choice has no effect on the results considered here and will be ignored.} of $\fg$.  Since $\ff$ is classical one can consider the cohomological support variety for a finite-dimensional $\ff$-module $M$. Moreover there is again a canonical isomorphism with the rank variety:
\[
\VV_{(\ff , \ff_{\0})}\left(M \right) \cong \VV^{\text{rank}}_{(\ff , \ff_{\0})}\left(M \right) = \left\{x \in  \ff_{\1}  \mid  \text{$M$ is not projective as a $U\!\left(\langle x \rangle \right)$-module} \right\} \cup \{0\}.
\]  We will identify the rank and support varieties for $\ff$.  Using the detecting subalgebra we can give the following geometric interpretation of the $z$-complexity of $\Delta(\mu)$.

\begin{theorem}\label{T:fsupportvariety}  Let $\mu \in X^{+}(T)$ be a dominant weight for  $\fp (n)$, set $h = \lfloor n/2 \rfloor$, and let $o=o(\mu)$ be the number of odd runs in $[\mu]$.  Then 
\[
\VV_{(\ff , \ff_{\0})}\left(\Delta(\mu) \right) = \left\{ \left. \sum_{i=1}^{h} a_{i}c_{i, h+i} \right| \text{$\# \left\{ i \mid a_{i} \neq 0 \right\} \leq \frac{n-o}{2}$} \right\}.
\] 
Moreover,
\[
\dim \VV_{(\ff , \ff_{\0})}\left(\Delta(\mu) \right) = \frac{n-o}{2}= c_{z}\left(\Delta(\mu) \right).
\]

\end{theorem}
\begin{proof}  As argued in the proof of \cite[Theorem 9.2.1]{BKN4}, the  $\Z$-grading on $\Delta(\mu)$ implies 
\[
\VV_{(\ff , \ff_{\0})}\left(\Delta(\mu) \right) = \ff_{\1} \cap \VV_{\fg_{-1}}\left(\Delta(\mu) \right).
\] Elements of $\ff_{\1} \cap \fg_{-1}$ are of the form $ \sum_{i=1}^{h} a_{i}c_{i, h+i}$ and have rank $2k$ if and only if precisely $k$ of the $a_{i}$ are nonzero.  Since $\VV_{\fg_{-1}}\left(\Delta(\mu) \right)$ is precisely the matrices of rank $n-o$ or less, the stated description of $\VV_{(\ff, \ff_{\0})}(\Delta(\mu))$ follows.   

As this variety has dimension $(n-o)/2$, the stated equalities will follow once we show the $z$-complexity of $\Delta(\mu)$ equals $(n-o)/2$.  By \cref{T:zcomplexity} $c_{z}(\Delta(\mu)) \leq (n-o)/2$.  However, if this were a strict inequality, then using this in the last displayed formula in the proof of \cref{T:complexityupperbound} would show the complexity of $\Delta(\mu)$ is strictly less than $\binom{n}{2} - \binom{o}{2}$, contradicting \cref{T:complexityandsupports}. 
\end{proof}

\let\section=\oldsection
\bibliographystyle{eprintamsmath}
\bibliography{BK1}

\end{document}